\documentclass{amsart}
\usepackage{amssymb,amsthm, amsmath, amsfonts}
\usepackage{graphics}
\usepackage{hyperref}
\usepackage[all]{xy}
\usepackage{enumerate}
\usepackage[mathscr]{eucal}

\theoremstyle{plain}
\newtheorem{theorem}{Theorem}[section]
\newtheorem{lemma}[theorem]{Lemma}
\newtheorem{proposition}[theorem]{Proposition}

\newtheorem{corollary}[theorem]{Corollary}
\numberwithin{equation}{section}

\theoremstyle{definition}

\newtheorem{definition}[theorem]{Definition}
\newtheorem{example}[theorem]{Example}
\newtheorem{remark}[theorem]{Remark}
\newtheorem{notation}[theorem]{Notation}

\DeclareMathOperator{\Ob}{Ob}
\DeclareMathOperator{\Mor}{Mor}
\DeclareMathOperator{\Id}{Id}
\DeclareMathOperator{\Module}{-Mod}

\DeclareMathOperator{\gMod}{-gMod}
\DeclareMathOperator{\gmod}{-gmod}
\DeclareMathOperator{\fdmod}{-gmod_{fd}}
\DeclareMathOperator{\lgmod}{-lgmod}
\DeclareMathOperator{\supp}{supp}

\DeclareMathOperator{\Hom}{Hom}

\DeclareMathOperator{\Ext}{Ext}
\DeclareMathOperator{\ini}{ini}
\DeclareMathOperator{\Ker}{Ker}
\DeclareMathOperator{\op}{op}

\DeclareMathOperator{\Fq}{\mathbb{F}}

\newcommand{\bk}{{\Bbbk}}
\newcommand{\C}{{\mathscr{C}}}
\newcommand{\tC}{{\underline{\mathscr{C}}}}
\newcommand{\D}{{\mathscr{D}}}
\newcommand{\tD}{{\underline{\mathscr{D}}}}
\newcommand{\tE}{{\underline{\mathcal{E}}}}
\newcommand{\F}{{\mathbb{F}}}
\newcommand{\cF}{{\mathscr{F}}}
\newcommand{\cFI}{{\mathscr{FI}}}

\newcommand{\cFS}{{\mathscr{FS}}}
\newcommand{\cOS}{{\mathscr{OS}}}
\newcommand{\cOI}{{\mathscr{OI}}}
\newcommand{\cVI}{{\mathscr{VI}}}
\newcommand{\II}{{\mathbb{I}}}

\newcommand{\tV}{{\underline{\mathscr{V}}}}
\newcommand{\tY}{{\underline{\mathscr{Y}}}}
\newcommand{\Z}{{\mathbb{Z}}}
\newcommand{\DD}{{\mathrm{D}}}
\newcommand{\Pb}{{P^{\bullet}}}
\newcommand{\FI}{{\mathrm{FI}}}
\newcommand{\trest}{{\downarrow _{\tD} ^{\tC}}}
\newcommand{\rest}{{\downarrow _{\tD} ^{\tC}}}

\newcommand{\Ai}{{\mathsf{A}_{\infty}}}

\title{Koszulity of directed categories in representation stability theory}
\author{Wee Liang Gan}
\address{Department of Mathematics, University of California, Riverside, CA 92521, USA}
\email{wlgan@math.ucr.edu}

\author{Liping Li}
\address{HPCSIP(Ministry of Education), College of Mathematics and Computer Science, Hunan Normal University, Changsha, Hunan 410081, China}
\email{lipingli@hunnu.edu.cn}

\thanks{The second author is supported by the National Natural Science Foundation of China 11771135 and the Start-Up Funds of Hunan Normal University 830122-0037. Both authors would like to thank the anonymous referee for carefully checking the paper and providing quite many valuable suggestions.}

\begin{document}

\begin{abstract}
In the first part of this paper, we study Koszul property of directed graded categories. In the second part of this paper, we prove a general criterion for an infinite directed category to be Koszul. We show that infinite directed categories in the theory of representation stability are Koszul over a field of characteristic zero.
\end{abstract}

\maketitle

\tableofcontents

\section{Introduction}

The goal of this paper is to prove, in a unified manner, the Koszulity of several infinite directed categories which appear recently in the theory of representation stability. Before giving an overview of our results, let us state the conventions which we shall use throughout the paper.

\subsection{Notations and conventions} \label{Conventions}
Let $\Z_+$ be the set of non-negative integers. For any $x\in \Z_+$, we write $[x]$ for the set $\{1,\ldots,x\}$.

Let $\bk$ be a field. We write $V^*$ for the dual space $\Hom_{\bk}(V, \bk)$ of a $\bk$-vector space $V$. For any set $S$, we shall write $\bk S$ for the $\bk$-vector space with basis $S$. By a \emph{graded} $\bk$-vector space, we mean a $\Z$-graded $\bk$-vector space. Morphisms of graded $\bk$-vector spaces are $\bk$-linear maps which do not necessarily preserve degrees. We shall write:

\begin{itemize}
\item $\bk\Module$ for the category of $\bk$-vector spaces;

\item $\bk\gMod$ for the category of graded $\bk$-vector spaces;

\item $\bk\gmod$ for the category of graded $\bk$-vector spaces with finite dimensional graded components.
\end{itemize}

By a \emph{category}, we shall always mean a small category. For any object $x$ of a category, we denote by $1_x$ the identity morphism of $x$.

We say a category is \emph{finite} when its set of objects is finite; otherwise, we say that it is \emph{infinite}.\footnote{This definition is non-standard since in most literature, a category is called \textit{finite} if it has only finitely many morphisms. However, it is more convenient for us since in our framework, we often work with $\bk$-linear categories with infinitely many morphisms when $\bk$ is an infinite ring, and it is the finiteness of objects instead of the finiteness of morphisms that plays the key role.} A \emph{$\bk$-linear category} is a category enriched over $\bk\Module$. To distinguish $\bk$-linear categories from usual categories, we shall denote them, for example, by $\tC$ rather than $\C$. For any category $\C$, we shall also write $\tC$ for its $\bk$-linearization, i.e. $\tC$ has the same set of objects as $\C$ and $\tC(x,y)=k\C(x,y)$ for all objects $x, y$.

A \emph{graded} $\bk$-linear category is a category enriched over $\bk \gMod$. For any graded $\bk$-linear category $\tC$, let
$$ \tC_i = \bigoplus_{x,y\in \Ob(\tC)} \tC(x,y)_i . $$
Then composition of morphisms defines the natural structure of a graded $\bk$-algebra on $\bigoplus_{i\in\Z} \tC_i$. By a \emph{positively graded} $\bk$-linear category, we shall always mean a graded $\bk$-linear category $\tC$ such that the following conditions are satisfied:

\begin{itemize}
\item[(P1)]
$\tC(x,y)$ is finite dimensional for all $x,y \in \Ob(\tC)$;

\item[(P2)]
$\tC(x,y)_i = 0$ for all $x,y \in \Ob(\tC)$ and $i<0$;

\item[(P3)]
$\tC(x,y)_0=0$ if $x\neq y$;

\item[(P4)]
for each $x\in \Ob(\tC)$, the $\bk$-algebra $\tC(x,x)_0$ is semisimple;

\item[(P5)]
for each $x \in \Ob(\tC)$, there are only finitely many objects $y$ such that
$\tC(x, y)_1 \neq 0$ or $\tC(y, x)_1 \neq 0$;

\item[(P6)]
$\tC$ is generated in degrees 0 and 1, in the sense that $\tC_1\cdot\tC_i = \tC_{i+1}$ for all $i\in\Z_+$;
\end{itemize}

A \emph{directed} $\bk$-linear category is a $\bk$-linear category $\tC$ and a partial order $\leqslant$ on $\Ob(\tC)$ such that for any $x, y \in \Ob \C$, one has $x \leqslant y$ whenever $\tC (x, y) \neq 0$.  A full subcategory $\tD$ of $\tC$ is said to be a \emph{convex} subcategory if for any $x, y, z\in \Ob(\tC)$ with $x \leqslant z \leqslant y$, one has $z\in\Ob(\tD)$ whenever $x,y\in \Ob(\tD)$. The \emph{convex hull} of a given set of objects of $\tC$ is the smallest convex full subcategory of $\tC$ containing this set of objects. By a \emph{directed graded} $\bk$-linear category, we shall always mean a directed positively graded $\bk$-linear category $\tC$ such that, in addition to (P1)--(P6), the following conditions are satisfied:

\begin{itemize}
\item[(P7)]
$\tC(x,x)_i = 0$ for all $x\in \Ob(\tC)$ and $i>0$;

\item[(P8)]
the convex hull of any finite set of objects contains only finitely many objects.
\end{itemize}

\subsection{Koszul theory for directed graded $\bk$-linear categories}
We shall recall  in Sections \ref{Preliminaries} and \ref{Koszul theory for graded directed categories} the background on Koszul theory for graded $\bk$-linear categories, following \cite{BGS} and \cite{Mazorchuk}.\footnote{A positively graded $\bk$-linear category in our sense is always Morita-equivalent to one in the sense of \cite{Mazorchuk}. However, the proof of Koszulity in our examples uses the underlying combinatorial structure of the category and so it is more convenient for us not to make the Morita-equivalent replacement.} The graded $\bk$-linear categories which we shall study in this paper are directed graded $\bk$-linear categories. Our first main results are the following.

\begin{theorem} \label{main-theorem-1}
Let $\tC$ be a directed graded $\bk$-linear category. Then $\tC$ is a Koszul category if and only if every finite convex full subcategory of $\tC$ is a Koszul category.
\end{theorem}

\begin{theorem} \label{main-theorem-2}
Let $\tC$ be a directed graded $\bk$-linear category. Let $\tD$ be the subcategory of $\tC$ defined by $\Ob(\tD) = \Ob(\tC)$ and
\begin{equation*}
 \tD(x,y)=\left\{ \begin{array}{ll}
 \tC(x,y) & \mbox{ if } x\neq y,\\
  \bk & \mbox{ if } x=y.
\end{array} \right.
\end{equation*}
Then $\tC$ is a Koszul category if and only if $\tD$ is a Koszul category.
\end{theorem}

\begin{remark} \label{essential subcategory}
Let us call the directed graded $\bk$-linear category $\tD$ defined in Theorem \ref{main-theorem-2} the \emph{essential subcategory} of $\tC$. It is immediate from Theorem \ref{main-theorem-2} that if two directed graded $\bk$-linear categories have isomorphic essential subcategories, then they are both Koszul when any one of them is Koszul. We emphasize that the conclusion holds because we are considering \emph{directed} categories; it fails for general $\bk$-linear categories or $\bk$-algebras.
\end{remark}

\subsection{Directed graded $\bk$-linear categories of type $\Ai$}
We say that a directed graded $\bk$-linear category $\tC$ is of \emph{type $\Ai$} if:

\begin{itemize}
\item $\Ob(\tC)=\Z_+$ with the natural partial order on $\Z_+$;

\item  for $x\leqslant y$, the graded vector space $\tC(x,y)$ is nonzero and concentrated in degree $y-x$.
 \end{itemize}
Let us give some examples of categories whose $\bk$-linearizations are directed graded $\bk$-linear categories of type $\Ai$.

\begin{example} \label{egFI}
Define the category $\cFI$ with $\Ob(\cFI)=\Z_+$ as follows. For any $x,y\in \Z_+$, let $\cFI(x,y)$ be the set of injections from $[x]$ to $[y]$. Then $\cFI$ is equivalent to the category $\FI$ of all finite sets and injections, which has appeared in many different contexts; in particular, see \cite[Section 1]{CEF}. The category of $\FI$-modules plays an important role in the theory of representation stability developed in \cite{CF}, \cite{CEF}, and \cite{CEFN} (see also \cite{Farb}). Let us also mention that the category $\cFI$  appears naturally in several other contexts; see, for example, \cite{KM} and \cite{RS}.
\end{example}

\begin{example} \label{egFIG}
(See \cite[Example 3.7]{GL}, \cite[Section 10.1]{SS}) Let $\Gamma$ be a finite group.  Define the category $\cFI_\Gamma$ with $\Ob(\cFI_\Gamma)=\Z_+$ as follows. For any $x,y\in \Z_+$, let $\cFI_\Gamma(x,y)$ be the set of all pairs $(f,c)$ where $f: [x] \to [y]$ is an injection, and $c: [x] \to \Gamma$ is an arbitrary map.  The composition of $(f_1, c_1)\in \cFI_\Gamma(x, y)$ and $(f_2, c_2)\in \cFI_\Gamma(y,z)$ is defined by $$ (f_2, c_2) (f_1, c_1) = (f_3, c_3) $$ where $$ f_3(r)=f_2(f_1(r)), \quad c_3(r)=c_2(f_1(r))c_1(r), \quad \mbox{ for all } r\in [x]. $$ Then $\cFI_{\Z/2\Z}$ is equivalent to the category $\FI_{BC}$ studied by J. Wilson \cite[Definition 1.1]{Wilson}.
\end{example}

\begin{example} \label{egFIoG}
Let $\Gamma$ be a finite abelian group. Define the subcategory $\cFI^\prime_\Gamma$ of $\cFI_\Gamma$ as follows. Let $\Ob(\cFI^\prime_\Gamma)=\Z_+$, and for any $x,y\in \Z_+$, let $\cFI^\prime_\Gamma(x,y)$ be the set of all pairs $(f,c) \in \cFI_\Gamma(x,y)$ such that $c(1)\cdots c(x)=1$ whenever $f$ is a bijection. Then $\cFI^\prime_{\Z/2\Z}$ is equivalent to the category $\FI_D$ studied by J. Wilson \cite[Definition 1.1]{Wilson}.
\end{example}

\begin{example} \label{egFIq}
(See \cite{Djament}, \cite[Example 3.9]{GL}, \cite[Section 1.2]{PS}) Let $\Fq$ be a finite field. Define the category $\cVI$ with $\Ob(\cVI)=\Z_+$ as follows. For any  $x,y\in \Z_+$, let $\cVI(x,y)$ be the set of injective linear maps from $\Fq^x$ to $\Fq^y$.
\end{example}

Suppose $\tC$ is a directed graded $\bk$-linear category of type $\Ai$. A $\tC$-module is a covariant $\bk$-linear functor $M: \tC\to \bk \Module$. We say that $M$ is \emph{generated in positions $\leqslant x$} if the only submodule of $M$ containing $M(y)$ for all $y\leqslant x$ is $M$.

The following theorem is the main result of this paper.

\begin{theorem} \label{main-theorem-3}
Let $\tC$ be a directed graded $\bk$-linear category of type $\Ai$. Suppose that there exists a faithful $\bk$-linear functor  $\iota: \tC \to \tC$ such that for each $x\in\Z_+$ one has $\iota(x)=x+1$, and the pullback of the $\tC$-module $\tC(x,-)$ via $\iota$ is a projective $\tC$-module generated in positions $\leqslant x$.  Then $\tC$ is a Koszul category.
\end{theorem}

It was first observed by T. Church, J. Ellenberg, B. Farb, and R. Nagpal \cite[Proposition 2.12]{CEFN} that the conditions in Theorem \ref{main-theorem-3} hold for the category $\cFI$. We shall give combinatorial conditions on $\C$ which will imply the conditions in the theorem. In practice, it is quite easy to verify those combinatorial conditions, and we shall show in Section \ref{combinatorial} that they hold for the categories $\cFI_\Gamma$ and $\cVI$ above, and the categories $\cOI_\Gamma$, $\cFI_d$, $\cOI_d$, $\cFS_G^{\op}$ and $\cOS_G^{\op}$ (see \cite{SS} and \cite{SS2}). Thus, when the characteristic of $\bk$ is 0, the $\bk$-linearizations of these categories are Koszul; moreover, from Theorem \ref{main-theorem-2}, we deduce that the $\bk$-linearization of $\cFI^\prime_\Gamma$ is also Koszul (when $\Gamma$ is abelian).

\subsection{Categories over $\cFI$}
Let $\C$ be a category and let $\rho : \C \to \cFI$ be a functor. We shall construct a directed graded $\bk$-linear category $\tC^{tw}$, called the \emph{twist} of $(\C, \rho)$. The twist $\tC^{tw}$ is constructed from $\tC$ by twisting the composition of morphisms in $\C$ by appropriate signs. We shall show that there is an equivalence $\mu$ from the category of $\tC^{tw}$-modules to the category of $\tC$-modules.

\begin{theorem} \label{main-theorem-4}
Assume that the characteristic of $\bk$ is 0. Let $\C$ be $\cFI_\Gamma$, $\cOI_\Gamma$, $\cFI_d$ or $\cOI_d$. Let $\tC^!$ be the quadratic dual of $\tC$, and let $\tY$ be the Yoneda category of $\tC$. Then one has the following:
\begin{enumerate}
\item
$\tC^!$ is isomorphic to $(\tC^{tw})^{\op}$.

\item
The category of $\tY$-modules is equivalent to the category of $\tC$-modules.

\item
The category $D^b(\tC\fdmod)$ is self-dual, where $\tC\fdmod$ denotes the category of graded $\tC$-modules which are finite dimensional.
\end{enumerate}
\end{theorem}
To prove Theorem \ref{main-theorem-4} (2), we note that by the Koszulity of $\tC$, we have $\tY \cong (\tC^!)^{\op}$, hence by (1) we have $\tY \cong \tC^{tw}$, and then we use the equivalence $\mu$ from the category of $\tC^{tw}$-modules to the category of $\tC$-modules to deduce that the category of $\tY$-modules is equivalent to the category of $\tC$-modules.

To prove Theorem \ref{main-theorem-4} (3), we construct an equivalence of categories
\begin{equation*}
 \kappa: D^b(\tC\fdmod) \longrightarrow D^b(\tC\fdmod)^{\op}
\end{equation*}
as a composition of the Koszul duality functor $\mathrm{K}$, vector space duality $(-)^*$, and $\mu$.

The Koszul property of the category of $\FI$-modules over a field of characteristic 0 was studied by S. Sam and A. Snowden in the language of twisted commutative algebras, see \cite{SS0}. In particular, in the case of the category FI, they constructed an equivalence of categories similar to $\kappa$ which they call the Fourier transform; it seems likely that these two functors are related.

\subsection{Structure of the paper}
The paper is organized as follows. In Section \ref{Preliminaries}, we recall the definition of a Koszul positively graded $\bk$-linear category. Then, in Section \ref{Koszul theory for graded directed categories}, we study Koszul theory for directed graded $\bk$-linear categories, and prove Theorem \ref{main-theorem-1} and Theorem \ref{main-theorem-2}. In Section \ref{type A-infinity}, we prove Theorem \ref{main-theorem-3}. In Section \ref{combinatorial}, we verify that our main examples satisfy the conditions in Theorem \ref{main-theorem-3}. In Section \ref{over FI}, we discuss the twist construction and prove Theorem \ref{main-theorem-4}.

\section{Modules over graded $\bk$-linear categories} \label{Preliminaries}

In this section, we introduce the various categories of modules over a graded $\bk$-linear category which we shall consider, and define the notion of a Koszul positively graded $\bk$-linear category. By a \emph{module}, we shall always mean a left-module.

\subsection{Graded $\bk$-linear categories and their modules}
Let $\tC$ be a graded $\bk$-linear category. We define:

\begin{itemize}
\item
$\tC\Module$, the category of \emph{$\tC$-modules}, as the category whose objects are the $\bk$-linear functors from $\tC$ to $\bk\Module$, and whose morphisms are the $\bk$-linear natural transformations of functors;

\item
$\tC\gMod$, the category of \emph{graded $\tC$-modules}, as the category whose objects are the degree-preserving $\bk$-linear functors from $\tC$ to $\bk\gMod$, and whose morphisms are the degree-preserving $\bk$-linear natural transformations of functors;

\item
$\tC\gmod$, the category of \emph{locally finite $\tC$-modules}, as the category whose objects are the degree-preserving $\bk$-linear functors from $\tC$ to $\bk\gmod$, and whose morphisms are the degree-preserving $\bk$-linear natural transformations of functors.
\end{itemize}

Thus, a $\tC$-module $M$ assigns a $\bk$-vector space $M(x)$ to each $x \in \Ob \tC$, and a linear transformation $M(\alpha): M(x) \to M(y)$ to each morphism $\alpha: x \to y$ in $\tC$, such that all composition relations in $\tC$ are respected. As usual, we shall write $\alpha$ for $M(\alpha)$.

\begin{example}
For any graded $\bk$-linear category $\tC$, and $x\in \Ob(\tC)$, the representable functor $\tC(x,-) : \tC \to \bk \gMod$ is a graded $\tC$-module.
\end{example}

\begin{remark}
If $\tC$ is the $\bk$-linearization of a category $\C$, then the notion of $\tC$-modules coincides with the notion of $\C$-modules over $\bk$, defined as functors from $\C$ to $\bk\Module$.
\end{remark}

For any  graded $\tC$-module $M$, let
$$  M_i = \bigoplus_{x\in \Ob(\tC)} M(x)_i . $$
Then $\bigoplus_{i\in\Z} M_i$ has the natural structure of a $\Z$-graded $\bigoplus_{i\in\Z} \tC_i$-module.

\begin{remark}
The functor $M\mapsto \bigoplus_{i\in\Z} M_i$ from the category $\tC\gMod$ to the category of $\Z$-graded $\bigoplus_{i\in\Z} \tC_i$-modules is fully faithful; it is essentially surjective if and only if $\tC$ is a finite category.
\end{remark}

For any graded $\tC$-module $M$ and $i\in\Z$, we define $\supp (M_i)$ to be the full subcategory of $\tC$ consisting of all objects $x \in \Ob(\tC)$ such that $M(x)_i \neq 0$, and define $\supp (M)$ to be the full subcategory of $\tC$ consisting of all objects $x \in \Ob(\tC)$ such that $M(x) \neq 0$. We say that $M$ is \emph{supported} on a full subcategory $\tD$ of $\tC$ if $\supp (M)$ is a full subcategory of $\tD$.

If $M$ is a $\tC$-module, we say that a subset $S \subseteq \bigcup_{x\in \Ob(\tC)}M(x)$ is a set of \textit{generators} of $M$ if the only submodule of $M$ containing $S$ is $M$ itself. The $\tC$-module $M$ is said to be \textit{finitely generated} if it has a finite set of generators. We say that a graded $\tC$-module $M$ is \textit{generated in degree} $i$ if $\bigcup_{x\in \Ob(\tC)}M(x)_i$ is a set of generators of $M$.

\begin{notation}[Internal degree shift]
If $M$ is a graded $\tC$-module, and $j\in \Z$, we define the graded $\tC$-module $M\langle j \rangle$ by $M\langle j\rangle(x)_i = M(x)_{i-j}$ for all $x\in \Ob(\tC)$ and $i\in \Z$.
\end{notation}

\begin{definition}
We say that a graded $\tC$-module $M$ satisfies condition:

\begin{itemize}
\item[(L1)] if there exists $n\in\Z$ such that $M_i=0$ for all $i<n$;

\item[(L2)] if for each $i\in\Z$, the $\bk$-vector space $M_i$ is finite dimensional.
\end{itemize}
We define $\tC\lgmod$, the category of \emph{lower bounded $\tC$-modules}, as the full subcategory of $\tC\gmod$ consisting of all graded $\tC$-modules $M$ satisfying both conditions (L1) and (L2).
\end{definition}

We note that the categories $\tC\Module$, $\tC\gMod$, $\tC\gmod$, and $\tC\lgmod$ are abelian.

\subsection{Projective covers}

Recall that a positively graded $\bk$-linear category is a graded $\bk$-linear category satisfying conditions (P1)--(P6) in Section \ref{Conventions}.

\begin{remark} \label{mazorchuk-positively-graded}
Our notion of positively graded $\bk$-linear category differs from the one in \cite[Definition 1]{Mazorchuk}: we do not require that $\tC(x,x)_0 = \bk$ for all $x$, but we require that $\tC$ is generated in degrees 0 and 1.
\end{remark}

\begin{lemma}  \label{finitely many y}
Let $\tC$ be a positively graded $\bk$-linear category. Let $x\in \Ob(\tC)$ and $i\ge 0$. Then there are only finitely many objects $y$ such that $\tC(x,y)_i \neq 0$ or $\tC(y,x)_i \neq 0$.
\end{lemma}
\begin{proof}
For $i=0$ or $1$, this follows from conditions (P3) and (P5). Suppose $i\ge 2$.
By condition (P6), one has
\begin{equation*}
\tC(x,y)_i = \sum_{z_2,\ldots, z_{i}\in \Ob(\tC)} \tC(z_i , y)_1 \cdots \tC(z_2,z_3)_1\cdot \tC(x,z_2)_1 .
\end{equation*}
It follows by condition (P5) that there are only finitely many $y$ such that $\tC(x,y)_i\neq 0$. Similarly, there are only finitely many $y$ such that $\tC(y,x)\neq 0$.
\end{proof}

\begin{lemma}
Let $\tC$ be a positively graded $\bk$-linear category. Then, for any $x\in \Ob(\tC)$, the $\tC$-module $\C(x,-)$ is lower bounded.
\end{lemma}
\begin{proof}
It is clear, by condition (P2), that $\tC(x,-)$ satisfies condition (L1). By condition (P1) and Lemma \ref{finitely many y}, it also satisfies condition (L2).
\end{proof}

The following corollary is immediate.
\begin{corollary}
Let $\tC$ be a positively graded $\bk$-linear category. If $M$ is a finitely generated graded $\tC$-module, then $M$ is lower bounded.
\end{corollary}

For any $x\in \Ob(\tC)$, the graded $\tC$-module $\tC(x,-)$ is a projective object in $\tC\gmod$. It follows by condition (P4) that if $W$ is any (ungraded) finite dimensional $\tC(x,x)_0$-module, the graded $\tC$-module $\tC(x,-)\otimes_{\tC(x,x)_0}W$ is also a projective object in $\tC\gmod$.

The following fundamental result is proved in the same way as \cite[Lemma 6(a)]{Mazorchuk}.

\begin{proposition} \label{projective-cover}
Let $\tC$ be a positively graded $\bk$-linear category. Let $M$ be a locally finite $\tC$-module. Suppose that for some $n\in\Z$, one has $M_i=0$ for all $i<n$. Then $M$ has a projective cover $\pi: P\to M$ in the category $\tC\gmod$. Moreover, $P_i=0$ for all $i<n$, and $\pi: P_n \to M_n$ is bijective. If $M$ is generated in degree $n$, then $P$ is generated in degree $n$. If $M$ is lower bounded, then $P$ is lower bounded.
\end{proposition}

\begin{proof}
By applying internal degree shift if necessary, one may assume that $n=0$. We claim that there is an epimorphism $\pi: V \to M$ in $\tC\gmod$ where $V$ is projective and $V_i=0$ for all $i<0$.

For each $x\in \Ob(\tC)$ and $i\ge 0$, we have the natural homomorphism
\begin{equation*}
(\tC(x,-)\otimes_{\tC(x,x)_0} M(x)_i )\langle i\rangle \longrightarrow M
\end{equation*}
whose image is the submodule of $M$ generated by $M(x)_i$. Taking direct sum, we obtain an epimorphism $\pi: V \to M$, where
\begin{equation*}
V = \bigoplus_{x\in \Ob(\tC),\, i\in\Z_+}  (\tC(x,-)\otimes_{\tC(x,x)_0} M(x)_i )\langle i\rangle .
\end{equation*}
Let $y\in\Ob(\tC)$ and $j\in\Z$. We have
\begin{equation*}
V(y)_j = \bigoplus_{x\in \Ob(\tC),\, 0\le i\le j} \tC(x,y)_{j-i} \otimes M(x)_i .
\end{equation*}
It follows by Lemma \ref{finitely many y} and condition (P1) that $\dim_{\bk} V(y)_j < \infty$. Hence, $V$ is a projective object in $\tC\gmod$, and $V_i=0$ for all $i<0$.

By Zorn's lemma, there exists a submodule $P$ of $V$ which is minimal with respect to the property that $\pi(P)=M$. It follows by a standard argument (see the proof of \cite[Proposition 1]{Sh}) that $P$ is a projective cover of $M$ in $\tC\gmod$. Moreover, $P_i = 0$ for all $i<0$. Since $\pi: V_0\to M_0$ is bijective, we must have $P_0=V_0$ and $\pi: P_0\to M_0$ is bijective. If $M$ is generated in degree $0$, then $\pi$ maps the submodule $P'$ of $P$ generated by $P_0$ surjectively onto $M$, so by minimality of $P$ we must have $P'=P$.

Suppose now that $M$ is lower bounded. Let $j\in \Z$. We have to show that $\dim_{\bk} P_j < \infty$. It suffices to show that $\dim_{\bk} V_j < \infty$. One has
\begin{equation*}
 V_j = \bigoplus_{i=0}^j \left( \bigoplus_{x,y\in \Ob(\tC)}  \tC(x,y)_{j-i} \otimes M(x)_i \right).
\end{equation*}
For each $i\le j$, there are only finitely many objects $x$ such that $M(x)_i\neq 0$, hence the result follows by Lemma \ref{finitely many y} and condition (P1).
\end{proof}

\begin{remark}
The category $\tC\gmod$ may not have enough projectives, see for example \cite[(1.3)]{Mazorchuk}.
\end{remark}

\begin{definition}
Let $\tC$ be a positively graded $\bk$-linear category and $M$ a locally finite $\tC$-module satisfying condition (L1). We define the \emph{syzygy} $\Omega M$ to be the kernel of the projective cover $P\to M$ of $M$ in $\tC\lgmod$.
\end{definition}

Since projective covers are unique up to isomorphism, one has:

\begin{corollary}
Let $\tC$ be a positively graded $\bk$-linear category and $M$ a locally finite $\tC$-module satisfying condition (L1). Then $M$ has a minimal projective resolution in $\tC\gmod$, and it is unique up to isomorphism. In particular, the syzygies $\Omega^i M$ for all $i\geqslant 0$ are well defined up to isomorphism in $\tC\gmod$. Moreover, if $M\in\tC\lgmod$, then its minimal projective resolution and $\Omega^i M$ for all $i\geqslant 0$ are also in $\tC\lgmod$.
\end{corollary}

Let us recall the definitions of Koszul module and Koszul category.

\begin{definition}
Let $\tC$ be a positively graded $\bk$-linear category. A graded module $M \in \tC \lgmod$ is \textit{Koszul} if it has a \textit{linear} projective resolution
\begin{equation*}
\ldots \to P^{-n} \to P^{-(n-1)} \to \ldots \to P^{-1} \to P^0 \twoheadrightarrow M;
\end{equation*}
that is, $P^{-i}$ is generated in degree $i$ for all $i \geqslant 0$. The category $\tC$ is said to be a \emph{Koszul category} if for every $x \in \Ob \tC$, $\tC (x, x)_0$ (considered as a graded $\tC$-module concentrated in degree 0) is a Koszul module.
\end{definition}

\begin{remark} \normalfont
It follows from Proposition \ref{projective-cover} that a graded module $M \in \tC \lgmod$ is Koszul if and only if for every $i \geqslant 0$, the syzygy $\Omega^i M$ is generated in degree $i$.
\end{remark}

\section{Koszul theory for directed graded $\bk$-linear categories}  \label{Koszul theory for graded directed categories}

There is an extensive literature on Koszul theory and we refer the reader to \cite{Mazorchuk} for a list of some of them. The purpose of this section is to study Koszul property in the special case of a directed graded $\bk$-linear category.

\subsection{Ideals and coideals}
Recall that a directed graded $\bk$-linear category is a directed positively graded $\bk$-linear category satisfying conditions (P7) and (P8).

Let $\tC$ be a directed graded $\bk$-linear category. Note that by conditions (P3) and (P7), $\tC$ is skeletal.

We call $x \in \Ob(\tC)$ a \textit{minimal} object if for any $y \in \Ob(\tC)$, one has either $x \leqslant y$ or $x$ and $y$ are incomparable. Dually, we define \emph{maximal} objects. A full subcategory $\tD$ is an \emph{ideal} of $\tC$ if for every $x \in \Ob(\tD)$ and every $y \in \Ob(\tC)$ with $y \leqslant x$, one has $y \in \Ob(\tD)$. Dually, we define \emph{coideals} of $\tC$. For a full subcategory $\tD$ of $\tC$, the full subcategory of $\tC$ consisting of all objects in $\Ob(\tC) \setminus \Ob (\tD)$ is called the \emph{complement} of $\tD$ and is denoted by $\tD^c$.

\begin{lemma} \label{ideal and coideal}
Let $\tC$ be a directed graded $\bk$-linear category. Then:
\begin{enumerate}
\item The complement of an ideal is a coideal; dually, the complement of a coideal is an ideal.
\item The intersection of a convex subcategory and an ideal is still convex; dually, the intersection of a convex subcategory and a coideal is still convex.
\end{enumerate}
\end{lemma}

\begin{proof}
(1) Let $\tD$ be an ideal of $\tC$. Take $x \in \Ob (\tD^c)$ and consider any $y \in \Ob (\tC)$ such that $x \leqslant y$. By definition, if $y \notin \Ob(\tD^c)$, then $y \in \Ob (\tD)$. But since $\tD$ is an ideal, every object less than or equal to $y$ must be contained in $\Ob (\tD)$. In particular, $x \in \Ob (\tD)$, a contradiction. Hence, $\tD^c$ is a coideal. The second statement is dual.

(2) Let $\tE$ be a convex subcategory and let $\tD$ be an ideal. Let $x, z \in \Ob (\tD) \cap \Ob (\tE)$. Take any $y \in \Ob(\tC)$ such that $x \leqslant y \leqslant z$. Since $\tE$ is convex, one has $y \in \Ob (\tE)$. Since $y \leqslant z$, $z \in \Ob \tD$, and $\tD$ is an ideal, we have $y \in \Ob (\tD)$. Therefore, $y \in \Ob (\tD )\cap \Ob (\tE)$. Hence, the intersection of $\tD$ and $\tE$ is convex. The second statement is dual.
\end{proof}

\begin{remark} \normalfont
Actually, as the reader may see, the above definitions and lemma do not involve the $\bk$-linear structure of the category. That is, they hold for arbitrary small categories whose underlying sets are equipped with a poset structure.
\end{remark}

Now suppose we have an ideal $\tD$ of a directed graded $\bk$-linear category $\tC$. Let $\tE=\tD^c$, and denote by $i :\tE \to \tC$ and $j: \tD \to \tC$ the inclusion functors. We have the restriction functors
\begin{gather*}
i^! : \tC\gMod \to \tE\gMod, \quad M \mapsto M\circ i , \\
j^* : \tC\gMod \to \tD\gMod, \quad M \mapsto M\circ j .
\end{gather*}
We also have the lifting functors
\begin{gather*}
i_* : \tE\gMod \to \tC\gMod \quad\mbox{ and }\quad j_* : \tD\gMod \to \tC\gMod
\end{gather*}
which regard a graded $\tE$-module or a graded $\tD$-module as a graded $\tC$-module in the obvious way. That is, for a graded $\tE$-module $M$ (respectively, a graded $\tD$-module $N$), the value of $i_*M$ (respectively, $j_*N$) on an object $x$ is $M(x)$ (respectively, $N(x)$) whenever $x$ is an object in $\tE$ (respectively, $\tD$), and is 0 otherwise. By the directed structure of $\tC$, $i_*M$ and $j_*N$ are $\tC$-modules. Furtheremore, it is clear that for any $\tE$-module $M$ and any $\tD$-module $N$, one has isomorphisms
$i^! i_* M \cong M$ and $j^*j_* N \cong N$.

\begin{lemma} \label{property of recollement}
One has:
\begin{enumerate}
\item $i_{\ast}$, $i^!$, $j^{\ast}$, and $j_{\ast}$ are exact functors.
\item $i^!$ is a right adjoint functor to $i_*$, while $j^*$ is a left adjoint functor to $j_*$.
\item $i_{\ast}$ and $j^{\ast}$ preserve projective modules.
\end{enumerate}
\end{lemma}

\begin{proof}
Part (1) is clear by the definitions of these functors. To show part (2), note that the directed structure of $\tC$ gives rise to a triangular matrix structure $(\tD, \tE, B)$, where we regard $\tD$ and $\tE$ as (non-unital) $\bk$-algebras, and
\begin{equation*}
B = \bigoplus _{x \in \tD, \, y \in \tE} \tC (x, y).
\end{equation*}
is a $(\tD, \tE)$-bimodule; see \cite{Li3}. Therefore, by a simple modification\footnote{The only trouble is that $\tC$ as a $\bk$-algebra has no identity. However, it still has a complete set of orthogonal idempotents.}  of \cite[Example 3.2 and Proposition 3.3]{Li3}, one can obtain a recollement of module categories, and the conclusion of part (2) follows from standard facts on recollements. Part (3) immediately follows from (1) and (2) since these two functors have exact right adjoint functors.
\end{proof}

\subsection{Restrictions to subcategories}

Let $\tC$ be a directed graded $\bk$-linear category.

Suppose $\tD$ is a graded $\bk$-linear subcategory of $\tC$ and $\iota: \tD \to \tC$ is the inclusion functor.  Then there is a pullback functor $\iota^{\ast}: \tC \gMod \to \tD \gMod$, sending $M \in \tC \gMod$ to $M \circ \iota$. The functor $\iota^{\ast}$ restricts to a functor $\tC \lgmod \to \tD \lgmod$. When $\tD$ is a subcategory of $\tC$ and $\iota$ is the inclusion functor, we denote $\iota^{\ast}$ by $\trest$.

The following proposition establishes a bridge between the Koszul property of directed graded $\bk$-linear categories with infinitely many objects and the Koszul property of directed graded $\bk$-linear categories with only finitely many objects.

\begin{proposition} \label{restriction of Koszul modules}
Let $\tC$ be a directed graded $\bk$-linear category. Let $M \in \tC \lgmod$. Then the following statements are equivalent:
\begin{enumerate}
\item $M$ is a Koszul $\tC$-module;
\item $M \trest$ is a Koszul $\tD$-module for every ideal $\tD$ of $\tC$;
\item $M \downarrow _{\tV} ^{\tC}$ is a Koszul $\tV$-module for every convex full subcategory $\tV$ containing $\supp (M_0)$.
\item $M \downarrow _{\tV} ^{\tC}$ is a Koszul $\tV$-module for every finite convex full subcategory $\tV$ containing $\supp(M_0)$.
\end{enumerate}
\end{proposition}

\begin{proof}
First of all, if $M$ is not generated in degree 0, it is easy to see that (1)--(4) are false. Thus, we may assume that $M$ is generated in degree 0.

Let $\tE$ be the minimal coideal of $\tC$ containing $\supp (M_0)$. Explicitly, $\Ob(\tE)$ is the set of all $y \in \Ob (\tC)$ such that $x\leqslant y$ for some $x \in \Ob(\supp(M_0))$.

It is easy to see that $M$ and all its syzygies are supported on $\mathcal{E}$. In other words, given a linear projective resolution $\Pb \to M$, applying the functor $i^! \cong \downarrow _{\tE} ^{\tC}$, we get a linear projective resolution $i^! \Pb \to i^! M$. Conversely, given a linear projective resolution of $i^!M$, applying the functor $i_{\ast}$, we get a linear projective resolution of $i_{\ast} i^! M \cong M$. Consequently, $M$ is Koszul if and only if $M \downarrow _{\tE} ^{\tC}$ is Koszul. Moreover, by Lemma \ref{ideal and coideal}, $\tV \cap \tE$ is a finite convex full category of $\tE$ containing $\supp(M_0)$, and it is straightforward to check that $\tE \cap \tD$ is an ideal of $\tE$. Therefore, without loss of generality we may assume that $\tC = \tE$.

$(1) \Rightarrow (2)$: Since $M$ is a Koszul $\tC$-module, we can find a (minimal) linear projective resolution $\Pb \to M$. Applying the functor $j^{\ast}$ defined above to this resolution, which according to Lemma \ref{property of recollement} is exact and preserves projective modules, we obtain a linear projective resolution $j^{\ast} \Pb \to j^{\ast} M$. Note that $j^{\ast}$ is isomorphic to $\trest$. The conclusion follows.

$(2) \Rightarrow (3)$: Under the assumption that $\tC = \tE$, we claim that $\tV$ is actually an ideal of $\tC$. Indeed, by our definition $\tC = \tE$ has \emph{enough minimal objects}; that is, for every $y \in \Ob (\tC)$, there is a minimal object $x \in \Ob (\tC)$ such that $x \leqslant y$. Note that all these minimal objects are contained in $\supp(M_0)$, and hence contained in $\tV$. Therefore, if there are objects $u \in \Ob (\tC)$ and $v \in \Ob (\tV)$ with $u \leqslant v$, we can find a minimal object $w \in \Ob (\tC)$ with $w \leqslant u$, so $w \leqslant u \leqslant v$. But both $v$ and $w$ are contained in $\tV$, which is convex. Thus $u \in \Ob (\tV)$ as well. This proves our claim. Now the conclusion follows from the argument of (1).

$(3) \Rightarrow (4)$: Trivial.

$(4) \Rightarrow (1)$: If $M$ is not Koszul, we want to find a finite convex full subcategory $\tV$ containing $\supp(M_0)$ such that $M \downarrow _{\tV} ^{\tC}$ is not Koszul. Let $\Pb \to M$ be a minimal projective resolution and define
\begin{equation*}
s = \inf \{ i \in \Z_+ \mid \Omega^i M \text{ is not generated in degree } i \}.
\end{equation*}
This $s$ exists since $M$ is not Koszul. Therefore, we can find certain $t > s$ and $z \in \Ob (\tC)$ such that $(\Omega^s M)(z)_t \neq 0$, and it is not contained in the submodule generated by $\bigoplus _{i < t} (\Omega^s M)_i$.

Now define $\tV$ to be the minimal convex subcategory containing
\begin{equation*}
\supp(P^0 / \bigoplus _{i > t} P^0_i) \bigcup \supp(P^{-1} / \bigoplus _{i > t} P^{-1}_i) \bigcup \ldots \bigcup \supp (P^{-(s-1)} / \bigoplus _{i > t} P^{-(s-1)}_i).
\end{equation*}
This is a finite category by condition (P8) since the above set is a finite. Moreover, as explained before, $\tV$ is actually an ideal of $\tC$. Note that $z \in \Ob(\tV)$.

Again, applying $\downarrow _{\tV} ^{\tC} \cong j^{\ast}$ one has
\begin{equation*}
j^{\ast} P^{-(s-1)} \to \ldots \to j^{\ast} P^0 \twoheadrightarrow j^{\ast} M.
\end{equation*}
Note that for $0 \leqslant i \leqslant s-1$, by our construction of $\tV$ and recursion one can show that $\Omega ^i (j^{\ast} M)$ is generated in degree $i$ for $i \leqslant s-1$, so $j^{\ast} P^{-i}$ is a projective cover of $\Omega ^i (j^{\ast} M)$. Moreover, for $i \leqslant s$, one has $j^{\ast} \Omega^i M \cong \Omega^i (j^{\ast} M)$. However, $\Omega ^s (j^{\ast} M)$ is not generated in degree $s$ since
\begin{equation*}
(\Omega ^s (j^{\ast} M))(z)_t \cong ( j^{\ast} \Omega^s M)(z)_t = (\Omega^s M)(z)_t
\end{equation*}
is not contained in the $\tV$-submodule generated by $(j^{\ast} \Omega^s M)_s = (\Omega^s M)_s$. Consequently, $j^{\ast} M$ is not a Koszul $\tV$-module.
\end{proof}

Correspondingly, one has the following result which contains Theorem \ref{main-theorem-1}:

\begin{theorem} \label{restriction of Koszul categories}
Let $\tC$ be a directed graded $\bk$-linear category. Then the followings are equivalent:
\begin{enumerate}
\item $\tC$ is a Koszul category;
\item every ideal of $\tC$ is a Koszul category;
\item every coideal of $\tC$ is a Koszul category;
\item every convex full subcategory of $\tC$ is a Koszul category;
\item every finite convex full subcategory of $\tC$ is a Koszul category.
\end{enumerate}
\end{theorem}

\begin{proof}
Let $\tD$ be a full subcategory of $\tC$. Note that $\tD$ is a Koszul category if and only if for every object $x \in \Ob (\tD)$, the graded $\tD$-module $\tD(x,x) = \tC (x, x)$ is a Koszul $\tD$-module. Let $M = \tC (x, x)$ and apply the previous proposition. One immediately deduces the equivalences of (1), (2), (4), and (5).

Obviously, (3) implies (1) since $\tC$ is a coideal of itself. If $\tD$ is a coideal, then for every $x \in \Ob \tD$, every term in a minimal projective resolution of the graded $\tC$-module $\tD (x, x)$ is supported on $\tD$. The functor $i_{\ast}$ identify $\tD$-modules with $\tC$-modules supported on $\tD$. Therefore, the $\tD$-module $\tD (x, x)$ has a minimal linear projective resolution if and only if the $\tC$-module $\tD (x, x)$ has a minimal linear projective resolution, so the Koszul property of $\tC$ implies the Koszul property of $\tD$.
\end{proof}

By reducing to the case of finite categories, we have:

\begin{proposition} \label{opposite category}
Let $\tC$ be a directed graded $\bk$-linear category. If $\tC$ is a Koszul category, then its opposite category $\tC ^{\op}$ is a Koszul category.
\end{proposition}

\begin{proof}
If $\tC$ is Koszul, then every finite convex full subcategory $\tV$ is Koszul. But $\tV$ can be considered as a finite dimensional graded algebra. By \cite[Proposition 2.2.1]{BGS}, $\tV ^{\op}$ is also Koszul. Since there is an obvious bijective correspondence between such convex subcategories $\tV$ of $\tC$ and $\tV ^{\op}$ of $\tC ^{\op}$, we deduce by Proposition \ref{restriction of Koszul categories} that
$\tC ^{\op}$ is a Koszul category.
\end{proof}

\subsection{Change of endomorphisms}
We shall give the proof of Theorem \ref{main-theorem-2} in this subsection.

\begin{theorem} \label{change endomorphisms}
Let $\tC$ be a directed graded $\bk$-linear category, and let $\tD$ be a graded $\bk$-linear subcategory of $\tC$. Suppose that they have the same objects and $\tC(x, y) = \tD(x,y)$ for $x, y \in \Ob \tC$ with $x \neq y$. For $M \in \tC \lgmod$, one has:
\begin{enumerate}
\item If $\tC$ is a Koszul category and $M$ is a Koszul $\tC$-module, then $M \trest$ is a Koszul $\tD$-module.
\item If $\tD$ is a Koszul category and $M \trest$ is a Koszul $\tD$-module, then $M$ is a Koszul $\tC$-module.
\end{enumerate}
\end{theorem}

\begin{proof}
By Proposition \ref{restriction of Koszul modules}, we can reduce to the case of finite categories. Indeed, if $\tV$ is a finite convex full subcategory of $\tC$ containing $\supp(M_0)$, then $\tV \cap \tD$ is a finite convex full subcategory of $\tD$ containing $\supp((M \trest)_0) = \supp (M_0)$, and
\begin{equation*}
\tV \leftrightarrow \tV \cap \tD
\end{equation*}
is a bijective correspondence. Furthermore, if $\tC$ is a Koszul category, so is $\tV$. Similarly, $\tV \cap \tD$ is a Koszul category when $\tD$ is.

By Proposition \ref{restriction of Koszul modules}, the Koszul property of $M$ implies the Koszul property of $M \downarrow ^{\C} _{\mathcal{V}}$ for each such $\tV$. Note that statement (1) holds for directed graded $\bk$-linear categories with only finitely many objects by \cite[Theorem 5.12]{Li2}. (In Theorems 5.12, 5.13, and 5.15 of \cite{Li2}, it was assumed that $\D(x, x)=k$ for $x \in \Ob \tC$  but the same proofs work in our case.) Thus, $(M \downarrow ^{\tC} _{\tV}) \downarrow ^{\tV} _{\tV \cap \tD}$ is Koszul for every such $\tV$. But
\begin{equation*}
M \downarrow ^{\tC} _{\tV} \downarrow ^{\tV} _{\tV \cap \tD} = M \downarrow _{\tV \cap \tD} ^{\tC} = M \trest \downarrow _{\tD \cap \tV} ^{\tD},
\end{equation*}
so $M \trest$ is still Koszul restricted to every $\tD \cap \tV$, which implies the Koszul property of $M \trest$ again by Proposition \ref{restriction of Koszul modules}.

Conversely, if $M \trest$ is Koszul, by Proposition \ref{restriction of Koszul modules}, so is $M \trest \downarrow _{\tD \cap \tV} ^{\tD} = M \downarrow ^{\tC} _{\tV} \downarrow ^{\tV} _{\tV \cap \tD}$ for every such $\tV$. Note that statement (2) holds for every finite category $\tV$ by \cite[Theorem 5.13]{Li2}. Thus every $M \downarrow _{\tV} ^{\tC}$ is Koszul. Again by Proposition \ref{restriction of Koszul modules}, $M$ is Koszul, too.
\end{proof}

We can now prove Theorem \ref{main-theorem-2}.

\begin{proof}[Proof of Theorem \ref{main-theorem-2}]
If $\tC$ is Koszul, then for every $x \in \Ob(\tC)$, $\tC (x, x)$ is a Koszul $\tC$-module. By (1) in the previous theorem, $\tC(x, x) \downarrow _{\tD} ^{\tC}$ is a Koszul $\tD$-module. But
\begin{equation*}
\tC(x, x) \downarrow _{\tD} ^{\tC} \cong \bigoplus _{i=1} ^{\dim_{\bk} \tC (x, x)} \tD (x,x).
\end{equation*}
Therefore, $\tD (x, x) = k1_x $ is a Koszul $\tD$-module. Consequently, $\tD$ is a Koszul category.

Conversely, if $\tD$ is a Koszul category, then for every $x \in \Ob(\tD)$, $\tD (x, x) = k1_x$ is a Koszul $\tD$-module, so is $\tC(x, x) \downarrow _{\tD} ^{\tC}$. By (2) in the previous theorem, $\tC (x, x)$ is a Koszul $\tC$-module. Consequently, $\tC$ is a Koszul category.
\end{proof}

\subsection{Quadratic categories}  \label{subsection on quadratic categories}

In this subsection, we show that if a directed graded $\bk$-linear category $\tC$ is Koszul, then it is quadratic. Note that results in this subsection (as well as the next subsection) can be deduced from \cite{Mazorchuk} using Morita theory. Indeed, since Koszulity is not a property of an algebra (or a $\bk$-linear category) but rather of its module category, Morita-equivalent algebras or $\bk$-linear categories are Koszul (or not) simultaneously. We note that in \cite{Mazorchuk}, an additional assumption is imposed on $\tC$: for every object $x \in \tC$, assume $\tC(x, x) = \bk$; however, one can always construct from $\tC$ (in our setting) a Morita-equivalent $\bk$-linear category $\overline{\tC}$ such that $\overline{\tC}$ satisfies this assumption. Nevertheless, we give detailed proofs of these results because: firstly, we can always reduce to the case of finite convex categories, a special advantage of directed categories; secondly, the main purpose of this paper is not to develop a general Koszul theory, but to describe a useful combinatorial criterion implying the Koszulity and apply it to many specific examples in representation stability theory -- if we reduce to the case described in \cite{Mazorchuk} using Morita theory, we will lose the nice combinatorial structure of $\tC$.

Let $\tC$ be a directed graded $\bk$-linear category. Note that $\tC_1$ is a $(\tC_0, \tC_0)$-bimodule. Thus we can define a tensor algebra
\begin{equation*}
\tC_0 [\tC_1] = \tC_0 \oplus \tC_1 \oplus (\tC_1 \otimes _{\tC_0} \tC_1) \oplus \ldots
\end{equation*}
We construct a category $\hat{\tC}$ which has the same objects as $\tC$ as follows: for $x, y \in \Ob \tC$, let $\hat {\tC} (x, y) = 1_y \big ( \tC_0 [\tC_1] \big ) 1_x$. We call $\hat {\tC}$ a \emph{free cover} of $\tC$.

\begin{lemma} \label{free cover}
Let $\tC$ be a directed graded $\bk$-linear category. Then the free cover $\hat {\tC}$ of $\tC$ is a directed graded $\bk$-linear category. Moreover, there is a degree-preserving quotient functor $\pi: \hat{\tC} \to \tC$ which is the identity map on the set $\Ob(\hat{\tC})$.
\end{lemma}

\begin{proof}
Obviously, $\hat{\tC}$ is directed with respect to the same partial order $\leqslant$. Also, the tensor structure of $\tC_0 [\tC_1]$ gives a natural grading on $\hat {\tC}$, where $\hat {\tC}_i=\tC_i$ for $i = 0$ and $i = 1$. Thus, conditions (P2)--(P8) are immediate. To prove condition (P1), take any two objects $x, y \in \tC$; we want to show that $\dim_{\bk} \hat {\tC} (x, y) < \infty$.

Let $\tV$ be the convex hull of $\{ x, y\}$ in $\hat {\tC}$. It is clear that $\hat {\tC} (x, y) = \tV (x, y)$. Moreover, the set of morphisms in $\tV$ is
\begin{equation*}
\tV_0 [\tV_1] = \tV_0 \oplus \tV_1 \oplus (\tV_1 \otimes _{\tV_0} \tV_1) \oplus \ldots
\end{equation*}
But by condition (P8), $\tV$ is a category with finitely many objects, so $\tV_0$ and $\tV_1$ are finite dimensional. Moreover, one has $\tV_1 ^{\otimes n} = 0$ for all sufficiently large $n$. Hence, $\tV_0 [\tV_1]$ is finite dimensional, and so is $\hat {\tC} (x, y) = \tV (x, y) = 1_y \big ( \tV_0 [\tV_1] \big ) 1_x$.

The definition of $\pi$ is straightforward. It is the identity map restricted to the sets of objects. Restricted to the sets of morphisms, it is induced by multiplication maps. Clearly, $\pi$ is a full functor, and is degree-preserving.
\end{proof}

Let $K$ be the kernel of $\pi: \hat {\tC} \to \tC$, which by definition is the subspace of $\C_0[\C_1]$ spanned by morphisms $\alpha$ in $\hat {\tC}$ such that $\pi (\alpha) = 0$. It is a $(\hat {\tC}, \hat {\tC})$-bimodule.

\begin{definition}
We say that $\tC$ is a \emph{quadratic} category if the kernel $K$ as a $(\hat {\tC}, \hat {\tC})$-bimodule has a set of generators contained in $\hat {\tC}_2 = \tC_1 \otimes _{\tC_0} \tC_1$; or equivalently, $K$ as a graded $(\hat {\tC}, \hat {\tC})$-bimodule is generated in degree 2.
\end{definition}

\begin{proposition}
The directed graded $\bk$-linear category $\tC$ is quadratic if and only if each convex full subcategory with finitely many objects is quadratic.
\end{proposition}

\begin{proof}
Suppose that $\tC$ is quadratic. Let $\tV$ be a convex full subcategory with finitely many objects. Let $1_{\tV}$ be $\sum _{x \in \Ob(\tV)} 1_x$. Restricted to objects in $\tV$, the quotient functor $\pi: \hat {\tC} \to \tC$ gives rise to a quotient functor $\pi_{\tV}: \hat {\tV} \to \tV$, whose kernel $K^{\tV}$ is precisely $1_{\tV} K 1_{\tV}$, where $K$ is the kernel of $\pi$. Since $\tC$ is quadratic, we have $K = \hat{\tC} K_2 \hat{\tC}$. Therefore, $K^{\hat{\tV}} = 1_{\tV} (\hat{\tC} K_2 \hat{\tC}) 1_{\tV}$. However, for every morphism $\alpha$ in $\hat{\tC}$ but not in $\hat{\tV}$, one has $1_{\tV} (\hat{\tC} \alpha \hat{\tC}) 1_{\tV} = 0$. This implies
\begin{equation*}
1_{\tV} (\hat{\tC} K_2 \hat{\tC}) 1_{\tV} = 1_{\tV} \big ( \hat{\tC} (K_2 \cap \hat{\tV}) \hat{\tC} \big) 1_{\tV} = 1_{\tV} \hat{\tC} 1_{\tV} (K_2 \cap \hat{\tV}) 1_{\tV} \hat{\tC}  1_{\tV} = \hat{\tV} K^{\hat{\tV}} _2 \hat{\tV}.
\end{equation*}
Therefore, $K^{\hat{\tV}}$ is generated in degree 2 as graded $(\hat {\tV}, \hat {\tV})$-bimodule, so $\hat {\tV}$ is quadratic.

Conversely, if every such $\tV$ is quadratic, then one has
\begin{equation*}
K^{\hat{\tV}} = \hat{\tV} K^{\hat{\tV}} _2 \hat{\tV} \subseteq \hat{\tC} K^{\hat{\tV}} _2 \hat{\tC}.
\end{equation*}
Since $K$ is the union of $K^{\hat{\tV}}$, it is generated by the union of all $K^{\hat{\tV}} _2$. That is, $\tC$ is quadratic.
\end{proof}

\begin{proposition} \label{Koszul categories are quadratic}
If a directed graded $\bk$-linear category $\tC$ is Koszul, then it is quadratic.
\end{proposition}

\begin{proof}
By the previous proposition and Proposition \ref{restriction of Koszul categories}, it suffices to prove the theorem for all convex full subcategories with only finitely many objects. But in this case, the result is well known, see \cite[Corollary 2.3.3]{BGS}.
\end{proof}

Let $\tC$ be a directed graded $\bk$-linear category. We shall define a directed graded $\bk$-linear category $\tC^!$, called the \emph{quadratic dual category} of $\tC$. The construction is well-known, see for example \cite[Section 4.1]{Mazorchuk}. We repeat it here for the convenience of the reader.

\begin{notation}
If $x,y \in \Ob(\tC)$, and $V$ is a left $\tC(y,y)$ and right $\tC(x,x)$ bimodule, we denote by $\DD V$ the left $\tC(x,x)$ and right $\tC(y,y)$ bimodule $\Hom_{\tC(y,y)} (V, \tC(y,y))$.
\end{notation}

Let $\DD \tC_1 = \bigoplus_{x,y\in \Ob(\tC)} \DD \tC(x,y)_1$. One can define a tensor algebra
\begin{equation*}
\tC_0 [\DD \tC_1] = \tC_0 \oplus \DD \tC_1 \oplus (\DD \tC_1 \otimes _{\tC_0} \DD \tC_1) \oplus \ldots
\end{equation*}
and hence a category $\check {\tC}$ with the same objects as $\tC$ by $\check {\tC} (x, y) = 1_y \big ( \tC_0 [\DD \tC_1] \big ) 1_x$. As in the proof of Lemma \ref{free cover}, one can show that $\check {\tC}$ is a directed graded $\bk$-linear category.

For $x, y, z \in \Ob \tC$, applying the contravariant functor $\DD$ to the composition map
\begin{equation*}
\tC (y, z)_1 \otimes _{\tC (y, y)} \tC (x, y)_1 \to \tC(x, z)_2 ,
\end{equation*}
we obtain (see \cite[Section 2.7]{BGS})
\begin{equation*}
\DD \tC (x, z)_2 \to \DD \big( \tC(y, z)_1 \otimes _{\tC (y, y)} \tC(x, y)_1 \big ) \cong \DD \tC (x, y)_1 \otimes _{\tC (y, y)} \DD \tC (y, z)_1.
\end{equation*}
Let $I$ be the $(\check {\tC}, \check{\tC})$-bimodule generated by images of all these dual maps. Finally, we define $\tC^!$ by
\begin{equation*}
\Ob(\tC^!)=\Ob(\tC) \quad\mbox{ and }\quad \tC^! (x, y) = \check {\tC} (x, y) / 1_y I 1_x .
\end{equation*}
The categories $\check {\tC}$ and $\tC^!$ are directed with respect to the opposite partial order $\leqslant ^{\op}$, rather than $\leqslant$.

\begin{remark}
It is well known that for any finite dimensional semisimple $\bk$-algebra $A$, if $M$ is an $A$-module, then there is an isomorphism $\Hom_A(M,A) \cong \Hom_{\bk}(M, \bk)$; see, for example, \cite[Proposition 2.7]{Broue}.
\end{remark}

\subsection{Koszul duality}

This subsection is a summary of some well-known results of Koszul duality theory following \cite{BGS} and \cite{Mazorchuk}. Although \cite{Mazorchuk} assumes that $\tC_0 (x, x) = k 1_x$ for $x \in \Ob \tC$, their proofs for the following results still hold under the weaker assumption that $\tC _0 (x,x)$ is a finite dimensional semisimple algebra for $x \in \Ob \tC$ (see \cite{BGS}).

Let $\tC$ be a directed graded $\bk$-linear category.

\begin{notation}[Cohomological degree shift]
As usual, we denote the shift functor on derived categories by $[1]$.
\end{notation}

The Yoneda category $\tY$ of $\tC$ is the graded $\bk$-linear category defined as follows. It has the same objects as $\tC$. For $x, y \in \Ob \tY$, we let (see \cite[Section 2.3]{Drozd} or \cite[Section 4.3]{Mazorchuk}):
\begin{align*}
\tY (x, y) & = \bigoplus _{i\in \Z} \left( \bigoplus_{j\in \Z} D (\tC \lgmod) (\tC (x, x), \tC (y, y) \langle j\rangle [i] ) \right) \\
& = \bigoplus_{i\in \Z} \left( \bigoplus_{j\in \Z} \Ext _{\tC \lgmod} ^i (\tC (x,x), \tC (y,y)\langle j \rangle) \right).
\end{align*}
In particular, if $\tC$ is Koszul, then
\begin{equation*}
\tY (x, y) = \bigoplus_{i \geqslant 0} \Ext _{\tC \lgmod} ^i (\tC (x,x), \tC (y,y)\langle i \rangle) .
\end{equation*}

\begin{proposition} \label{yoneda vs quadratic dual}
\cite[Proposition 17]{Mazorchuk}
Let $\tC$ be a directed graded $\bk$-linear category which is Koszul.
 Then there is an isomorphism of graded $\bk$-linear categories $\tY \cong (\tC^!)^{\op}$.
\end{proposition}

\begin{proof}
Following the notations in  \cite[Proposition 17]{Mazorchuk}, one has a natural $\Z$-action on $\Ext _{\tC} ^{\mathrm{Lin}} (\mathrm{L})$ (the full subcategory of $D (\tC \lgmod)$ consisting of all objects of the form $\tC (x, x)  \langle i \rangle [i]$ for $x \in \Ob \tC$ and $i \in \Z$), and a natural $\Z$-action on $((\tC^!) ^{\Z}) ^{\op}$. Moreover, these two actions are compatible. Taking the quotient categories modulo these $\Z$-actions, we recover $\tY$ and $(\tC^!)^{\op}$. Thus, the isomorphism of $\Ext _{\tC} ^{\mathrm{Lin}} (\mathrm{L})$ and $(\tC^!) ^{\Z}) ^{\op}$ in  \cite[Proposition 17]{Mazorchuk} induces an isomorphism of their quotients by the $\Z$-actions.
\end{proof}

We shall need the following lemma.

\begin{lemma}
Let $\tC$ be a directed graded $\bk$-linear category which is Koszul. Let $M \in \tC \lgmod$ be a Koszul $\tC$-module and let $\tV$ be a finite full convex subcategory of $\tC$ containing $\supp (M_0)$. Then, for each $i \geqslant 0$, one has
\begin{equation*}
1_{\tV} \Ext _{\tC \gmod} ^{i} (M, \tC_0\langle i \rangle) \cong \Ext _{\tV \gmod}^i  (1_{\tV} M, \tV_0\langle i \rangle) ,
\end{equation*}
where $1_{\tV} = \sum_{x\in \Ob(\tV)} 1_x$. (On the left hand side, $1_{\tV}$ is an element of $\tY_0$, whereas on the right hand side, it is an element of $\tC_0$.)
\end{lemma}

\begin{proof}
By taking a minimal linear projective resolution $\Pb \to M$, we deduce that
\begin{equation*}
\Ext ^i _{\tC \gmod} (M, \tC_0\langle i \rangle) \cong \Hom _{\tC \gmod} (\Omega ^i M, \tC_0\langle i \rangle).
\end{equation*}
Note that $\Omega^i M$ is generated in degree $i$. Therefore,
\begin{equation*}
\Hom _{\tC \gmod} (\Omega ^i M, \tC_0\langle i \rangle) \cong \Hom _{\tC \gmod} ((\Omega ^i M)_i, \tC_0\langle i \rangle).
\end{equation*}
Thus,
\begin{align*}
1_{\tV} \Ext ^i _{\tC \gmod} (M, \tC_0\langle i \rangle) & \cong 1_{\tV} \Hom _{\tC \gmod} ((\Omega ^i M)_i, \tC_0 \langle i \rangle)\\
& \cong \Hom_{\tC \gmod} ((\Omega ^i M)_i, \tV_0\langle i \rangle)\\
& \cong \Hom_{\tC \gmod} (1_{\tV} (\Omega ^i M)_i, \tV_0\langle i \rangle)\\
& \cong \Hom _{\tV \gmod} (1_{\tV} (\Omega ^i M)_i, \tV_0\langle i \rangle).
\end{align*}

As explained in the proof of Proposition \ref{restriction of Koszul modules}, the restriction functor $\downarrow _{\tV} ^{\tC}$ is exact and preserves projective modules, so we get a minimal projective resolution $1_{\tV} \Pb \to 1_{\tV} M$. Consequently, $\Omega^i (1_{\tV} M) \cong 1_{\tV} \Omega ^i M$. Combining this with the previous isomorphism, we have:
\begin{align*}
1_{\tV} \Ext ^i _{\tC \gmod} (M, \tC_0\langle i \rangle) &\cong \Hom _{\tV \gmod} (\Omega ^i (1_{\tV} M)_i, \tV_0\langle i \rangle)\\
& \cong \Ext _{\tV \gmod} ^i (1_{\tV} M, \tV_0\langle i \rangle) .
\end{align*}
\end{proof}

We have the following result, which is well-known when the category is finite.

\begin{theorem}
Let $\tC$ be a directed graded $\bk$-linear category which is Koszul. Then its Yoneda category $\tY$ is Koszul, and the Yoneda category of $\tY \cong (\tC^!)^{\rm{op}}$ is isomorphic to $\tC$. Moreover, the functors
\begin{equation*}
\mathrm{E} = \bigoplus_{i \geqslant 0}\Ext^{i} _{\tC \gmod} (-, \tC_0\langle i \rangle)
\quad \mbox{ and } \quad \mathrm{F} = \bigoplus_{i\geqslant 0} \Ext^{i} _{\tY \gmod} (-, \tY_0\langle i \rangle)
\end{equation*}
give anti-equivalences between the full subcategory of Koszul $\tC$-modules in $\tC \lgmod$ and the full subcategory of Koszul $\tY$-modules in $\tY \lgmod$.
\end{theorem}

\begin{proof}
We use Proposition \ref{restriction of Koszul modules}, Theorem \ref{restriction of Koszul categories}, and the previous lemma to reduce to the case of finite categories.

Let $M \in \tC \lgmod$ be a Koszul $\tC$-module. Then $\mathrm{E} M$ is a graded $\tY$-module. Using the isomorphism
\begin{equation*}
\Ext ^i _{\tC \gmod} (M, \tC_0\langle i \rangle) \cong \Hom _{\tC \gmod} ((\Omega ^i M)_i, \tC_0\langle i \rangle) ,
\end{equation*}
one sees that $(\Omega^i M)_i$ and $(\mathrm{E} M)_i$ are supported on the same objects. In particular, $\dim_{\bk} (\mathrm{E} M)_i < \infty$ for $i \geqslant 0$ since $\Omega^i M \in \tC \lgmod$. Thus, $\mathrm{E} M \in \tY \lgmod$.

Let $\tV$ be a finite convex full subcategory of $\tC$ containing $\supp (M_0)$. By Theorem \ref{restriction of Koszul categories}, the category $\tV$ is Koszul. The Yoneda category $\tY_{\tV}$ of $\tV$ is a finite convex full subcategory of $\tY$ containing $\supp ((\mathrm{E} M)_0) = \supp(M_0)$. By Proposition \ref{restriction of Koszul modules}, $M \downarrow _{\tV} ^{\tC}$ is a Koszul $\tY$-module. Therefore, by above lemma, one has
\begin{align*}
\mathrm{E} M \downarrow _{\tY _{\tV}} ^{\tY} & = \bigoplus_{i\geqslant 0} 1_{\tV} \Ext _{\tC \gmod} ^{i} (M, \tC_0\langle i \rangle)\\
& \cong \bigoplus_{i\geqslant 0} \Ext _{\tV \gmod} ^{i} (1_{\tV}M, \tV_0\langle i \rangle),
\end{align*}
which is a Koszul $\tY _{\tV}$-module since $\tV$ is a finite category, see for example \cite[Theorem 1.2.5]{BGS}. Note that $\tV \mapsto \tY _{\tV}$ defines a bijective correspondence between finite convex full subcategories containing $\supp (M_0)$. Therefore, by Proposition \ref{restriction of Koszul modules}, $\mathrm{E}M$ is a Koszul $\tY$-module. In particular, taking $M$ to be $\tC (x, x)$ for $x \in \Ob \tC$, one deduces that $\tY$ is a Koszul category. It follows from Proposition \ref{yoneda vs quadratic dual} that the Yoneda category of $\tY$ is isomorphic $\tC$.

To show $M \cong \mathrm{FE} M$, we define $\tV_n$, for $n\ge 0$,  to be the convex hull of
\begin{equation*}
\supp (\bigoplus _{i=0}^n  M_i) \cup \supp (\bigoplus _{i=0}^n  (\mathrm{FE} M)_i).
\end{equation*}
Since the category $\tV_n$ contains only finitely many objects, it follows (for example from \cite[Theorem 4.1]{Li2}) that there is an isomorphism $\varphi_n: 1_{\tV_n} M  \cong 1_{\tV_n} (\mathrm{FE} M)$. Moreover, these isomorphisms are compatible with the restriction functors.
Thus, there is an isomorphism $\varphi: M \to \mathrm{FE} M$ such that the restriction of $\varphi$ to $\tV_n$ is $\varphi_n$ for $n\ge 0$. Similarly, one has $N \cong \mathrm{EF} N$ for any Koszul module $N\in \tY\lgmod$.
\end{proof}

Following \cite{BGS}, we let $C^{\uparrow} (\tC \gmod)$ be the category of complexes $M^{\bullet}$ of graded $\tC$-modules such that there exist $r, s \in \Z$ satisfying
\begin{equation*}
M^i_j = 0 \text{ if } i>r \text{ or } i+j < s.
\end{equation*}
Dually, we define $C^{\downarrow} (\tC \gmod)$ be the category of complexes $M^{\bullet}$ of graded $\tC$-modules such that there exist $r, s \in \Z$ satisfying
\begin{equation*}
M^i_j = 0 \text{ if } i<r \text{ or } i+j > s.
\end{equation*}
Let $D^{\uparrow} (\tC \gmod)$ and $D^{\downarrow} (\tC \gmod)$ be their corresponding derived categories.

We omit the proof of the following derived equivalence which is the same as \cite[Theorem 30]{Mazorchuk} (see also \cite[Theorem 2.12.1]{BGS}).

\begin{theorem}  \label{koszul derived equivalence}
Let $\tC$ be a directed graded $\bk$-linear category. If $\tC$ is Koszul, then there exists an equivalence of triangulated categories
\begin{equation*}
 \mathrm{K}: D^{\downarrow} (\tC \gmod)  \longrightarrow D^{\uparrow} (\tC^! \gmod) .
\end{equation*}
\end{theorem}

We refer the reader to \cite[(5.6)]{Mazorchuk} for the construction of the equivalence (see also the proof of  \cite[Theorem 2.12.1]{BGS}).

\begin{remark}
The notations for $D^\uparrow$ and $D^\downarrow$ in \cite{Mazorchuk} are opposite to the ones in \cite{BGS}; we have followed the notations in \cite{BGS}. Thus, the above functor $\mathrm{K}$ is the functor which \cite{Mazorchuk} denotes by $\mathrm{K}'$.
\end{remark}

\section{Koszulity  criterion in type $\Ai$} \label{type A-infinity}

\emph{In this section, we assume that $\tC$ is a directed graded $\bk$-linear category of type $\Ai$.}

\subsection{Koszul duality for directed graded $\bk$-linear categories of type $\Ai$}

 We say that a $\tC$-module or a $\tC^!$-module $M$ is \emph{generated in position $x$} if $M$ is generated by $M(x)$. For example, for any $x\in \Z_+$, the graded $\tC$-module $\tC(x,-)$ is generated in position $x$.

\begin{lemma} \label{degree and position}
(1) Let $M$ be a graded $\tC$-module which is generated in degree $i$ for some $i\in\Z$, and suppose that $M_i \subset M(x)$ for some $x\in \Z_+$. Then $M_{i+n}=M(x+n)$ for all $n\in \Z$.

(2) Let $N$ be a graded $\tC^!$-module which is generated in degree $i$ for some $i\in\Z$, and suppose that $N_i \subset N(x)$ for some $x\in \Z_+$. Then $N_{i+n}=N(x-n)$ for all $n\in \Z$.

\end{lemma}
\begin{proof}
(1) One has:
\begin{equation*}
M_{i+n} = \tC_n M_i \subset \tC_n M(x) = \tC(x,x+n) M(x) \subset M(x+n) .
\end{equation*}
But
\begin{equation*}
\bigoplus_{n\in\Z} M_{i+n} = \bigoplus_{n\in\Z} M(x+n) ,
\end{equation*}
so $M_{i+n}=M(x+n)$ for each $n$.

(2) The proof is similar to (1). Note that $\tC^!_n N(x) = \tC^!(x,x-n) N(x)$.
\end{proof}

\begin{lemma} \label{koszul by position}
Let $M\lgmod$ be a lower bounded $\tC$-module generated in degree 0 and suppose that $M_0=M(x)$ for some $x\in \Z_+$. Then $M$ is Koszul if and only if for every $n \geqslant 0$, the syzygy $\Omega^n M$ is generated in position $x+n$.
\end{lemma}

\begin{proof}
Recall that $M$ is Koszul if and only if for every $n \geqslant 0$, the syzygy $\Omega^n M$ is generated in degree $n$. Let us first prove the following statement.

\medskip

{\it Claim}: If $\Omega^n M$ is generated in degree $n$ and $(\Omega^n M)_n = (\Omega^n M)(x+n)$, then $(\Omega^{n+1} M)_{n+1} = (\Omega^{n+1} M)(x+n+1)$.

\medskip

To prove the claim, let $\pi: P\to \Omega^n M$ be a projective cover of $\Omega^n M$. By Proposition \ref{projective-cover}, $P$ is generated in degree $n$, and $\pi: P_n \to (\Omega^n M)_n$ is bijective. Since $(\Omega^n M)_n = (\Omega^n M)(x+n)$, it follows that $P_n \subset P(x+n)$, and hence by Lemma \ref{degree and position} one has $P_{n+1} = P(x+n+1)$ and $(\Omega^n M)_{n+1} = (\Omega^n M)(x+n+1)$. Therefore,
\begin{align*}
(\Omega^{n+1} M)_{n+1} &= \Ker (\pi : P_{n+1}\to (\Omega^n M)_{n+1}) \\ &= \Ker (\pi : P(x+n+1) \to (\Omega^n M)(x+n+1) ) \\ &= (\Omega^{n+1} M)(x+n+1).
\end{align*}

Now we return to the proof of the lemma. If $M$ is Koszul, then it follows  by induction on $n$ using the above claim that $\Omega^n M$ is generated in position $x+n$ for all $n\ge 0$. If $M$ is not Koszul, there is a \emph{minimal} $m \ge 1$ such that $\Omega^m M$ is not generated by $(\Omega^m M)_m$, and it follows by induction on $n$ using the above claim that $(\Omega^n M)_n = (\Omega^n M)(x+n)$ for $n=1,\ldots, m$. This implies that $\Omega^m M$ is not generated by $(\Omega^m M)(x+m)$.
\end{proof}

\begin{notation}
For any $\tC$-module $M$, let
\begin{equation*}
\ini (M) = \inf \{x \in \Z_+ \mid M(x) \neq 0 \},
\end{equation*}
where by convention $\inf \emptyset = \infty$.

For any $\tC^!$-module $N$, let
\begin{equation*}
\ini (N) = \sup \{x \in \Z_+ \mid N(x) \neq 0 \},
\end{equation*}
where by convention $\sup \emptyset = -\infty$.
\end{notation}

\begin{lemma} \label{initial position}
(1) Let $M \in \tC \lgmod$ be a lower bounded $\tC$-module. Then one has
\begin{equation*}
\ini(\Omega^{n+1} M) \geqslant \ini(\Omega^n M) + 1 \quad\mbox{ for each }n\in \Z_+.
\end{equation*}
(2) Let $N \in \tC \lgmod$ be a lower bounded $\tC^!$-module. Then one has
\begin{equation*}
\ini (\Omega^{n+1} N) \leqslant \ini (\Omega^n N) -1  \quad\mbox{ for each }n\in \Z_+.
\end{equation*}
\end{lemma}

\begin{proof}
(1) Let $\pi: P\to \Omega^n M$ be the projective cover of $\Omega^n M$. Let $x=\ini (\Omega^n M)$. Since $P$ has no strictly smaller submodule whose image is $\Omega^n M$, it follows that $\pi: P(x)\to (\Omega^n M)(x)$ must be bijective. Hence $(\Omega^{n+1} M)(x)=0$.

(2) The proof is completely similar to (1).
\end{proof}

A graded $\tC$-module or a graded $\tC^!$-module $M$ is said to be \emph{finite dimensional} if $\bigoplus_{i\in\Z} M_i$ is finite dimensional. We denote by $\tC\fdmod$ (respectively $\tC^!\fdmod$) the category of graded $\tC$-modules (respectively graded $\tC^!$-modules) which are finite dimensional.

We note that for any graded $\tC^!$-module $N$:
\begin{equation*}
\mbox{ $N$ is lower bounded } \Leftrightarrow \mbox{ $N$ is finitely generated } \Leftrightarrow \mbox{ $N$ is finite dimensional. }
\end{equation*}

\begin{corollary} \label{finite projective dim}
Let $N$ be a finite dimensional graded $\tC^!$-module. Then $N$ has finite projective dimension.
\end{corollary}
\begin{proof}
By Lemma \ref{initial position}, one has $\Omega^n N = 0$ for all $n$ sufficiently large.
\end{proof}

The next corollary is immediate from Theorem \ref{koszul derived equivalence} and Corollary \ref{finite projective dim}; see the proof of \cite[Theorem 2.12.6]{BGS}.

\begin{corollary}  \label{koszul duality functor on bounded derived categories}
Suppose that $\tC$ is Koszul. Then the equivalence in Theorem \ref{koszul derived equivalence} induces an equivalence of triangulated categories
\begin{equation*}
 \mathrm{K} : D^b(\tC\fdmod) \longrightarrow D^b(\tC^!\fdmod).
\end{equation*}
\end{corollary}

\begin{remark}
We point out that the functor $K$ in the above theorem commutes with the cohomological degree shift functor (defined in Notation 3.14) in the derived category, but does not commute with the internal degree shift functor (defined in Notation 2.4); see \cite[Theorem 22 (ii)]{Mazorchuk}. Moreover, the category of finitely generated graded $\tC$-modules is not contained in $D^{\downarrow}(\tC\gmod)$. In \cite[Theorem 2.12.6]{BGS}, in their notations, a finitely generated $A$-module is always finitely generated over $\bk$; on the other hand, in our case, a finitely generated $\tC^!$-module is always finite dimensional.
\end{remark}

\subsection{Koszulity criterion}
In this subsection, we prove our main result, Theorem \ref{main-theorem-3}. We remind the reader that $\tC$ is a directed graded $\bk$-linear category of type $\Ai$.

If a $\bk$-linear functor $\iota: \tD \to \tC$ is injective on the set of objects and faithful, and $M$ is a $\tC$-module, we denote by $M\rest$ the $\tD$-module $M\circ\iota$ and call it the \emph{restriction} of $M$ to $\tD$.

\begin{definition}  \label{genetic functor}
We call a functor $\iota: \tD \to \tC$ a \emph{genetic functor} if it is a $\bk$-linear functor that satisfies the following conditions:

\begin{itemize}
\item[(F1)]
$\tD$ is a directed graded $\bk$-linear category of type $\Ai$;

\item[(F2)]
$\iota$ is faithful and $\iota(x)=x+1$ for all $x\in \Z_+$;

\item[(F3)]
for each $x\in \Z_+$, the $\tD$-module $\tC(x,-)\rest$ is projective and generated in positions $\leqslant x$.
\end{itemize}
\end{definition}

\begin{remark}
Suppose that $M$ is a module of a directed graded $\bk$-linear category of type $\Ai$. We say that $M$ is generated in positions $\leqslant x$ if $\bigcup_{y\leqslant x} M(y)$ is a set of generators of $M$.
\end{remark}

We refer the reader to \cite[Definition 2.7]{CEFN} for an example of a genetic functor from the category $\cFI$ to itself, see \cite[Proposition 2.12]{CEFN}.\footnote{The positive degree shift functor $S_{+a}$ in \cite[Definition 2.8]{CEFN} for $a=1$ is the restriction functor $\rest$.} The key observation of our present paper is that conditions (F1)--(F3) when $\tD=\tC$ allow one to prove by an induction argument that $\tC$ is Koszul.

The following lemma explains our choice of name for the functor.
\begin{lemma}
Let $\iota: \tD \to \tC$ be a genetic functor. If $P\in \tC\gmod$ is a projective graded $\tC$-module generated in positions $\leqslant x$ for some $x\in\Z_+$, then $P\rest$ is a projective graded $\tD$-module generated in positions $\leqslant x$.
\end{lemma}
\begin{proof}
Let $y\leqslant x$. By condition (F3), the restriction of any direct summand of $\tC(y,-)$ to $\tD$ is projective and generated in positions $\leqslant y$. The lemma follows.
\end{proof}

\begin{lemma} \label{degree shift}
Let $\iota: \tD \to \tC$ be a genetic functor and let $M$ be a $\tC$-module. Suppose that $\ini(M)\geqslant 1$. If the $\tD$-module $M\rest$ is generated in position $x$ for some $x\in\Z_+$, then the $\tC$-module $M$ is generated in position $x+1$.
\end{lemma}

\begin{proof}
For any $y\in \Z_+$, one has $M(y+1)=M\rest(y)$, and hence
\begin{equation*}
M(y+1) = \iota(\tD(x,y))M(x+1) \subset \tC(x+1,y+1) M(x+1) .
\end{equation*}
\end{proof}

\begin{lemma}  \label{crucial lemma}
Let $\iota: \tD \to \tC$ be a genetic functor. Let $M \in \tC \lgmod$ be a lower bounded $\tC$-module which is generated in position $x\in \Z_+$. Then one has:
\begin{equation*}
(\Omega M) \rest \cong \Omega (M \rest) \oplus Q
\end{equation*}
where $Q$ is a projective graded $\tD$-module generated in position $x$. In particular, $(\Omega M) \rest$ is generated in position $x$ if and only if $\Omega (M \rest)$ is generated in position $x$.
\end{lemma}

\begin{proof}
Let $P \to M$ be the projective cover of $M$. Then $P$ is generated in position $x$. Applying the restriction functor $\rest$ to the short exact sequence
\begin{equation*}
0 \to \Omega M \to P \to M \to 0 ,
\end{equation*}
we obtain a short exact sequence
\begin{equation*}
0 \to (\Omega M) \rest \to P \rest \to M \rest \to 0.
\end{equation*}
Hence,
\begin{equation*}
(\Omega M) \rest \cong \Omega (M \rest) \oplus Q,
\end{equation*}
where $Q$ is isomorphic to a direct summand of $P \rest$. But $P \rest$ is a projective graded $\tD$-module generated in positions $\leqslant x$, so $Q$ is also a projective graded $\tD$-module generated in positions $\leqslant x$.

However, by Lemma \ref{initial position},
\begin{equation*}
\ini (\Omega M) \geqslant \ini (M) +1 = x+1,
\end{equation*}
which implies
\begin{equation*}
\ini ((\Omega M) \rest) \geqslant x.
\end{equation*}
Thus, $Q(y)=0$ for all $y<x$, so $Q$ is generated in position $x$.
\end{proof}

\begin{proposition}  \label{restriction is koszul implies it is koszul}
Let $\iota: \tD \to \tC$ be a genetic functor. Let $M \in \tC \lgmod$ be a lower bounded $\tC$-module which is generated in degree 0 and suppose that $M_0 = M(x)$ for some $x \geqslant 1$. Then $M$ is a Koszul $\tC$-module if $M \rest$ is a Koszul $\tD$-module.
\end{proposition}

\begin{proof}
Suppose that $M \rest$ is a Koszul $\tD$-module. Then it is generated in degree 0 and one has
\begin{equation*}
(M \rest)_0 = M_0 = M(x) = M \rest (x-1).
\end{equation*}
Thus, by Lemma \ref{koszul by position}, for each $n \geqslant 0$, the syzygy $\Omega^n  (M \rest)$ is generated in position $x+n-1$.

To show that $M$ is Koszul, it suffices (by Lemma \ref{koszul by position}) to prove that $\Omega ^n M$ is generated in position $x+n$ for all $n\in\Z_+$. Clearly, this holds for $n=0$. Now let $n>0$. Suppose, for induction, that $\Omega^{r} M$ is generated in position $x+r$ for all $r < n$.

Applying Lemma \ref{crucial lemma} to the $\tC$-module $\Omega^{r}M$ for $r=0,\ldots,n-1$, one has
\begin{equation} \label{Qr}
(\Omega^{r+1} M)\rest \cong \Omega( (\Omega^r M) \rest) \oplus Q^r
\end{equation}
where $Q^r$ is a projective graded $\tC$-module generated in position $x+r$. Then, applying $\Omega^{n-r-1}$ to both sides of (\ref{Qr}) for $r=0,\ldots,n-2$, one has
\begin{equation} \label{omega-n}
\Omega^{n-r-1}( (\Omega^{r+1} M) \rest) \cong \Omega^{n-r} ( (\Omega^r M)\rest ) .
\end{equation}
It follows from (\ref{omega-n}) that
\begin{equation} \label{omega-nnn}
\Omega( (\Omega^{n-1} M)\rest ) \cong \Omega^n( M\rest ) .
\end{equation}
Hence, from (\ref{Qr}) for $r=n-1$, we have
\begin{align*}
(\Omega^n M)\rest & \cong \Omega ( (\Omega^{n-1} M)\rest ) \oplus Q^{n-1} \\
& \cong \Omega^n( M\rest ) \oplus Q^{n-1} \qquad \mbox{ (using (\ref{omega-nnn}))}.
\end{align*}
Since $\Omega^n( M\rest )$ and $Q^{n-1}$ are both generated in position $x+n-1$, it follows that $(\Omega^n M)\rest$ is generated in position $x+n-1$. But $\ini(\Omega^n M)\geqslant 1$, so
by Lemma  \ref{degree shift}, we deduce that $\Omega^n M$ is generated in position $x+n$.
\end{proof}

\begin{proposition} \label{G0 is koszul}
Let $\iota: \tD \to \tC$ be a genetic functor. Then $\tC(0,0)$ is a Koszul $\tC$-module.
\end{proposition}
\begin{proof}
Let $W=\tC(0,0)$ considered as a graded $\tC$-module. The projective cover of $W$ is $\tC(0,-) \to W$. Since $W \rest = 0$, one has
\begin{equation*}
  (\Omega W)\rest \cong \tC(0,-)\rest.
\end{equation*}
Hence, $(\Omega W)\rest$ is a projective graded $\tD$-module generated in position 0. By Lemma \ref{degree shift}, $\Omega W$ is generated in position 1. Note that $(\Omega W)(1) = (\Omega W)_1$. So $(\Omega W)\rest$ is a projective graded $\tD$-module generated in degree 1.

Let $M=(\Omega W) \langle -1 \rangle$. Then $M(1) = M_0$ and so $M$ is generated in degree 0. Moreover, $M\rest$ is a projective graded $\tD$-module generated in degree 0, so $M\rest$ is a Koszul $\tD$-module. It follows by Proposition \ref{restriction is koszul implies it is koszul} that $M$ is a Koszul $\tC$-module, and hence $W$ is a Koszul $\tC$-module.
\end{proof}

We can now prove Theorem \ref{main-theorem-3}, which is part of the following result.

\begin{theorem} \label{self embedding and Koszul}
Let $\iota: \tD \to \tC$ be a genetic functor. Then $\tC$ is Koszul when any one of the following hold:
\begin{enumerate}
\item $\tD$ is Koszul;
\item $\tD=\tC$.
\end{enumerate}
\end{theorem}

\begin{proof}
For each $x\in\Z_+$, let $W^x = \tC(x,x)$ regarded as a graded $\tC$-module, and let $U^x = \tD(x,x)$ regarded as a graded $\tD$-module. We already know, by Proposition \ref{G0 is koszul}, that $W^0$ is Koszul. It remains to show that $W^x$ is Koszul for $x\geqslant 1$.

Suppose $x\geqslant 1$. By Proposition \ref{restriction is koszul implies it is koszul}, it suffices to prove that $W^x \rest$ is Koszul. However, $W^x \rest$ is a direct summand of $(U^{x-1})^{\oplus a_x}$ for some $a_x\geqslant 1$.

In case (1), since $U^{x-1}$ is Koszul, it follows that $W^x\rest$ is Koszul and hence $W^x$ is Koszul.

In case (2), since $U^{x-1}=W^{x-1}$ and $W^0$ is Koszul, it follows by an induction argument that $W^x\rest$ is Koszul and hence $W^x$ is Koszul.
\end{proof}

\section{Combinatorial conditions} \label{combinatorial}

The purpose of this section is to give combinatorial conditions which imply the existence of a genetic functor $\iota: \tC \to \tC$. We shall verify these conditions for several examples.

\emph{In this section, we assume that $\bk$ is a field of characteristic 0.}

\subsection{Combinatorial conditions}  \label{subsection on combinatorial conditions}
A category is called an \emph{EI category} when  all endomorphisms in the category are isomorphisms.\footnote{As far as we are aware, the notion of EI category was defined by tom Dieck \cite{Dieck} and W. L\"{u}ck \cite{Luck}.}

In this subsection, suppose that $\C$ is an EI category satisfying the following conditions:

\begin{itemize}
\item[(E1)]
$\Ob(\C)=\Z_+$ or $\Z_+\setminus\{0\}$;

\item[(E2)]
$\C(x,y)$ is an empty set if $x>y$;

\item[(E3)]
$\C(x,y)$ is a nonempty finite set if $x\leqslant y$;

\item[(E4)]
for $x\leqslant y\leqslant z$, the composition map
\begin{equation*}
\C(y,z) \times \C(x,y) \to \C(x,z)
\end{equation*}
is surjective.
\end{itemize}
The $\bk$-linearization $\tC$ of $\C$ is graded with $\tC(x,y)$ concentrated in degree $y-x$ for $x\leqslant y$. It is clear that $\tC$ is isomorphic to a directed graded $\bk$-linear category of type $\Ai$.

To avoid confusion, we shall use the following notation.
\begin{notation}
Let $\II\in \C(1,1)$ be the identity morphism $1_1$ of $1\in\Ob(\C)$.
\end{notation}

\begin{proposition} \label{factorization}
Suppose $\C$ satisfies the following conditions:
\begin{itemize}
\item[(C1)]
There is a monoidal structure $\odot$ on $\C$ with
\begin{equation*}
x \odot y = x+y \quad \mbox{ for all } x,y\in\Ob(\C) .
\end{equation*}

\item[(C2)]
For all $x,y \in \Ob(\C)$, the map
\begin{equation*}
\C(x,y) \to \C(1+x, 1+y), \quad f \mapsto \II \odot f
\end{equation*}
is injective.

\item[(C3)]
Each morphism $f \in \C(x,1+y)$ has a factorization into a composition of
\begin{equation*}
f_1: x \to  1+z \qquad \mbox{ and } \qquad \II \odot f_2 : 1+z \to 1+y
\end{equation*}
where $z\leqslant x$ and $f_2:z\to y$;  moreover, for each $f$, if $z$ is the minimal integer for which such a factorization exists, then $f_2$ is unique given $f_1$, and $f_1$ is unique up to automorphisms of $z$.

\item[(C4)]
If a morphism $f \in \C(x,1+y)$ does not possess a factorization described in condition (C3) with $z<x$, then for any morphism $f_3 \in \C(y,y')$, the composition $(\II\odot f_3)\circ f\in \C(x,1+y')$ also does not possess a factorization described in (C3) with $z<x$.
\end{itemize}
Let $\iota: \tC \to \tC$ be the $\bk$-linear functor defined by $x\mapsto 1+x$ for any object $x$ and $f\mapsto \II\odot f$ for any morphism $f$ of $\C$. Then $\iota: \tC \to \tC$ is a genetic functor, and $\tC$ is Koszul.
\end{proposition}

\begin{proof}
Condition (F1) clearly holds. Condition (F2) is immediate from condition (C2). We have to verify condition (F3) with $\tD=\tC$. It suffices to show that for any $x\in\Ob(\C)$, there is a decomposition
\begin{equation}  \label{special decomposition}
 \tC(x,-) \rest \cong \tC(x,-)^{\oplus m_x} \oplus \tC(x-1,-)^{\oplus n_x}
\end{equation}
for some $m_x, n_x \geqslant 0$.

First, it is plain that in condition (C3), $z$ must be $x$ or $x-1$. Let $\C(x,1+x)'$ be the subset of $\C(x,1+x)$ consisting of all $f$ such that there is no factorization described in condition (C3) with $z=x-1$ (and $y=x$).

Now let $x,y\in \Ob(\C)$, and let $G_z=\C(z,z)$ for any $z\in\Ob(\C)$. For $z=x$, we choose a set of representatives $\beta_1, \ldots, \beta_m$ for the set of orbits $G_z \backslash \C(x,1+z)'$. For $z=x-1$, we choose a set  of representatives $\gamma_1, \ldots, \gamma_n$ for the set of orbits $G_z \backslash \C(x,1+z)$.

For $1\le r\le m$, and $y \ge x$, we define a map
\begin{equation*}
\C(x,y) \to \C(x,1+y), \quad \alpha \mapsto (\II \odot \alpha) \circ \beta_r,
\end{equation*}
and extend it linearly to a morphism
\begin{equation*}
\Theta_r : \tC(x,-) \to \tC(x,-)\rest .
\end{equation*}

Similarly, for $1 \le s\le n$, and $y \ge x-1$, we define a map
\begin{equation*}
\C(x-1,y) \to \C(x, 1+y), \quad \alpha \mapsto (\II \odot \alpha) \circ \gamma_s
\end{equation*}
and extend it linearly to a morphism
\begin{equation*}
\Theta'_s: \tC(x-1,-) \to \tC(x,-)\rest .
\end{equation*}

Let
\begin{eqnarray*}
 \Theta: \tC(x,-)^{\oplus m} \oplus \tC(x-1,-)^{\oplus n} & \to & \tC(x,-)\rest, \\
(\alpha_1,\ldots, \alpha_m, \alpha'_1, \ldots, \alpha'_n) & \mapsto &
\sum_{r=1}^m \Theta_r (\alpha_r) + \sum_{s=1}^n \Theta'_s (\alpha'_s).
\end{eqnarray*}
By conditions (C1)--(C4), the morphism $\Theta$ is bijective. This proves (\ref{special decomposition}).

By Theorem \ref{self embedding and Koszul}, it follows that $\tC$ is Koszul.
\end{proof}

For the category $\cFI$, the decomposition (\ref{special decomposition}) was proved in \cite[Proposition 2.12]{CEFN}. T. Church, J. Ellenberg, B. Farb., and R. Nagpal used their \cite[Proposition 2.12]{CEFN} to prove that a finitely generated $\cFI$-module over a Noetherian ring is Noetherian; it appears to be crucial in their proof that in the case of $\cFI$, the number $m_x$ in (\ref{special decomposition}) is 1.

\subsection{Examples} We now give examples of categories satisfying conditions (C1)--(C3) in Proposition \ref{factorization}. In all these example, we define the tensor product $\odot$ on objects by $x\odot y = x+y$ for all $x,y$.

\begin{example}[The category $\cFI_\Gamma$] \label{egfigamma}
Recall the category $\cFI_\Gamma$ defined in Example \ref{egFIG}. If $(f_1,c_1)\in \cFI_\Gamma(x_1,y_1)$ and $(f_2,c_2)\in \cFI_\Gamma(x_2,y_2)$, we define $(f_1,c_1)\odot (f_2,c_2)$ to be the morphism $(f,c)\in \cFI_\Gamma(x_1+x_2,y_1+y_2)$ where
\begin{gather*}
f(r) = \left\{  \begin{array}{ll}
f_1(r) & \mbox{ if } r\leqslant x_1,\\
f_2(r-x_1)+y_1 & \mbox{ if } r> x_1.
\end{array}\right.
\end{gather*}
and
\begin{gather*}
c(r) = \left\{  \begin{array}{ll}
c_1(r) & \mbox{ if } r\leqslant x_1,\\
c_2(r-x_1) & \mbox{ if } r> x_1.
\end{array}\right.
\end{gather*}
It is clear that conditions (C1) and (C2) are satisfied. We now check condition (C3). Let $(f,c)\in \cFI_\Gamma(x,1+y)$.

Suppose first that $f(m)=1$ for some $m\in[x]$. Let $z=x-1$ and let $f_1:[x]\to [x]$ be any bijection with $f_1(m)=1$. Define $f_2: [z]\to [y]$ by $f_2(f_1(r)-1) =f(r)$ for all $r\in[x]\setminus \{m\}$. Let $c_1: [x]\to \Gamma$ be a map such that $c_1(m)=c(m)$. Define $c_2: [z]\to \Gamma$ by $c_2(f_1(r)-1)= c(r) c_1(r)^{-1}$ for all $r\in [x]\setminus\{m\}$.

Now suppose that $f(m)\neq 1$ for all $m\in [x]$. Then there is no factorization with $z=x-1$. Let $z=x$, and let $f_1: [x]\to [1+x]$ be any injection whose image is $\{2,\ldots, 1+x\}$. Define $f_2:[z]\to [y]$ by $f_2(f_1(r)-1)=f(r)$ for all $r\in[x]$. Let $c_1:[x]\to \Gamma$ be any map. Define $c_2:[z]\to \Gamma$ by $c_2(f_1(r)-1)= c(r) c_1(r)^{-1}$.

In both cases, one has $(f,c)= (\II \odot (f_2,c_2))(f_1,c_1)$. Moreover, it is clear that $(f_2,c_2)$ is unique given $(f_1,c_1)$, and $(f_1,c_1)$ is unique up to automorphisms of $z$. Condition (C4) is also clear from the above observations.
\end{example}

\begin{example}[The category $\cOI_\Gamma$] \label{egoig}
Let $\Gamma$ be a finite group and define $\cOI_\Gamma$ to be the subcategory of $\cFI_\Gamma$ with $\Ob(\cOI_\Gamma) = \Z_+$ as follows: for any $x, y\in \Z_+$, the set $\cOI_\Gamma$ consists of all $(f,c)\in \cFI_\Gamma$ such that $f$ is increasing.

It is clear that conditions (C1) and (C2) are satisfied. To check condition (C3), note that given $(f,c)\in \cOI_\Gamma(x,y)$, we can choose its factorization in $\cFI_\Gamma$ with $f_1$ a unique increasing map; then, $f_2$ is also an increasing map. Condition (C4) is clear.
\end{example}

\begin{notation}
If $f:[x]\to [y]$, we write $\Delta_f$ for the set $[y]\setminus \mathrm{Im}(f)$.
\end{notation}

\begin{example}[The category $\cFI_{d}$] \label{egfid}
Let $d$ be a positive integer. We define the category $\cFI_{d}$ following \cite[Section 7.1]{SS}. Let $\Ob(\cFI_d)=\Z_+$. For any $x, y \in \Z_+$, let $\cFI_d(x,y)$ be the set of all pairs $(f,\delta)$ where $f:[x]\to [y]$ is an injection, and $\delta: \Delta_f \to [d]$ is an arbitrary map. The composition of $(f_1,\delta_1) \in \cFI_d(x,y)$ and $(f_2,\delta_2)\in \cFI_d(y,z)$ is defined by $(f_2,\delta_2) (f_1, \delta_1) = (f_3, \delta_3)$ where $f_3=f_2f_1$ and
$$ c_3 (m)= \left\{ \begin{array}{ll}
c_1(r) & \mbox{ if } m=f_2(r) \mbox{ for some } r, \\
c_2(m) & \mbox{ else. } \end{array} \right. $$
We define a monoidal structure on $\cFI_d$ similarly to Example \ref{egfigamma}. It is clear that conditions (C1) and (C2) are satisfied. We now check condition (C3). Let $(f,\delta)\in C(x,1+y)$.

Suppose first that $f(m)=1$ for some $m\in[x]$. Let $z=x-1$ and let $f_1:[x]\to [x]$ be any bijection with $f_1(m)=1$. Define $f_2: [z]\to [y]$ by $f_2(f_1(r)-1) =f(r)$ for all $r\in[x]\setminus\{m\}$. Let $\delta_1 :\emptyset \to [d]$ be the unique map. Define $\delta_2 : \Delta_{f_2}\to [d]$ by $\delta_2(m) = \delta(m+1)$ for all $m\in \Delta_{f_2}$.

Now suppose that $f(m)\neq 1$ for all $m\in [x]$. Then there is no factorization with $z=x-1$. Let $z=x$, and let $f_1: [x]\to [1+x]$ be any injection whose image is $\{2,\ldots, 1+x\}$. Define $f_2:[z]\to [y]$ by $f_2(f_1(r)-1)=f(r)$ for all $r\in[x]$. Define $\delta_1 : X_{f_1} \to [d]$ by $\delta_1(1)=\delta(1)$. Define $\delta_2 : \Delta_{f_2} \to [d]$ by $\delta_2(m) = \delta(m+1)$ for all $m\in \Delta_{f_2}$.

In both cases, one has $(f,\delta)= (\II\odot (f_2,\delta_2))(f_1,\delta_1)$. Moreover, it is clear that $(f_2,\delta_2)$ is unique given $(f_1,\delta_1)$, and $(f_1,\delta_1)$ is unique up to automorphisms of $z$. Condition (C4) is also clear from the above observations.
\end{example}

\begin{remark}
The category of $\underline{\cFI_d}$-modules is equivalent to the category of modules over a certain \emph{twisted commutative algebra}, see \cite[Proposition 7.3.4]{SS}.
\end{remark}

\begin{example}[The category $\cOI_d$] \label{egoid}
Let $d$ be a positive integer. We define the subcategory $\cOI_d$ of $\cFI_d$ following \cite[Section 7.1]{SS}. Let $\Ob(\cOI_d) = \Z_+$. For any $x,y\in\Z_+$, let $\cOI_d(x,y)$ be the set of all pairs $(f,\delta)\in \cFI_d$ such that $f$ is increasing.

It is clear that conditions (C1) and (C2) are satisfied. To check condition (C3), note that given $(f,\delta)\in \cOI_d(x,y)$, we can choose its factorization in $\cFI_d$ with $f_1$ a unique increasing map; then, $f_2$ is also an increasing map. Condition (C4) is clear.
\end{example}

\begin{example}[The opposite category of $\cFS_\Gamma$] \label{egfsgamma}
Let $\Gamma$ be a finite group. We define the category $\cFS_\Gamma$ following \cite[Section 10.1]{SS}. Let $\Ob(\cFS_\Gamma)=\Z_+\setminus\{0\}$. For any $x, y\in \Z_+\setminus\{0\}$, let $\cFS_\Gamma(y,x)$ be the set of all pairs $(f, c)$ where $f: [y]\to [x]$ is a surjection, and $c:[y]\to \Gamma$ is an arbitrary map. The composition of $(f_1, c_1)\in \cFS_\Gamma(y, x)$ and $(f_2, c_2)\in \cFS_\Gamma(z,y)$ is defined by $$ (f_1, c_1) (f_2, c_2) = (f_3, c_3) $$ where $$ f_3(r)=f_1(f_2(r)), \quad c_3(r)=c_1(f_2(r))c_2(r), \quad \mbox{ for all } r\in [z] . $$
We define a monoidal structure on $\cFS_\Gamma^{\op}$ similarly to Example \ref{egfigamma}. It is clear that conditions (C1) and (C2) are satisfied. We now check condition (C3). Let $(f,c)\in \cFS_\Gamma(1+y, x)$. Let $m=f(1)$.

Suppose first that $f^{-1}(m)=\{1\}$. Let $z=x-1$ and let $f_1:[x] \to [x]$ be any bijection such that $f_1(1)=m$. Define $f_2: [y]\to [z]$ by $f_2(r) = f_1^{-1}(f(r+1)) -1$ for all $r\in [y]$. Let $c_1: [1+z] \to \Gamma$ be any map such that $c_1(1)=c(1)$. Define $c_2: [y] \to \Gamma$ by $c_2(r) = c_1(f_2(r)+1)^{-1}c(r+1)$ for all $r\in [y]$.

Now suppose that $f^{-1}(m) \neq \{1\}$. Then there is no factorization with $z=x-1$. Let $z=x$. Let $f_1: [1+z] \to [x]$ be any surjection such that $f_1(1)=m$ and $f_1^{-1}(m) \neq \{1\}$. Define $f_2 : [y] \to [z]$ by $f_2(r) = n$ if $f_1(n+1) = f(r+1)$. Again, let $c_1: [1+z] \to \Gamma$ be any map such that $c_1(1)=c(1)$, and define $c_2: [y] \to \Gamma$ by $c_2(r) = c_1(f_2(r)+1)^{-1}c(r+1)$ for all $r\in [y]$.

In both cases, one has $(f,c)=(f_1,c_1) (\II\odot (f_2,c_2))$ in $\cFS_\Gamma$. Moreover, it is clear that $(f_2,c_2)$ is unique given $(f_1,c_1)$, and $(f_1,c_1)$ is unique up to automorphisms of $z$. Condition (C4) is also clear from the above observations.
\end{example}

\begin{example}[The opposite category of $\cOS_\Gamma$] \label{egosg}
Let $\Gamma$ be a finite group.  We define the subcategory $\cOS_\Gamma$ of $\cFS_\Gamma$ following \cite[Section 8.1]{SS}. Let $\Ob(\cOS_\Gamma) = \Z_+\setminus\{0\}$. For any $x,y\in\Z_+\setminus\{0\}$, let $\cOS_\Gamma(y,x)$ be the set of all pairs $(f,c)\in  \cFS_\Gamma(y, x)$ where $f$ is an \emph{ordered surjection}, in the sense that for all $r<s$ in $[x]$ we have $\mathrm{min} \,f^{-1}(r) < \mathrm{min}\, f^{-1}(s)$.

Similarly to above, it is clear that conditions (C1) and (C2) are satisfied, and to check condition (C3), we note that given $(f,c)\in \cOS_\Gamma(y,x)$, we can choose its factorization in $\cFS_\Gamma$ for a unique choice of ordered surjections $f_1$ and $f_2$. Condition (C4) is clear.
\end{example}

\begin{example}[The category $\cVI$] \label{egvi}
Recall the category $\cVI$ defined in Example \ref{egFIq}. If $f_1\in \cVI(x_1,y_1)$ and $f_2\in \cVI(x_2,y_2)$, we define $f_1\odot f_2\in \cVI(x_1+x_2,y_1+y_2)$ by $f_1\odot f_2 = f_1 \oplus f_2 : \Fq^{x_1} \oplus \Fq^{x_2} \to \Fq^{y_1} \oplus \Fq^{y_2}$.  It is clear that conditions (C1) and (C2) are satisfied. We now check condition (C3).

Let $f \in \C(x,1+y)$. We write $f$ as a $(1+y)\times x$-matrix. Let $u$ be the first row of $f$, and $h$ be the $y \times x$-matrix form by the last $y$ rows of $f$. Suppose that we have a factorization $f = (\II \odot f_2) \circ f_1$ for some $z$. Then the first row of $f_1$ must be $u$. Let $p$ be the $z\times x$-matrix form by the last $z$ rows of $f_1$. Then $h = f_2 p$. Since $f_2$ is injective, one has
\begin{equation*}
\mbox{(rank of $h$)} = \mbox{(rank of $p$)} \leqslant z.
\end{equation*}

Now if we first choose $h = f_2 p$ to be any factorization of $h$ into the composition of a surjective linear map $p: \F^x\to \F^z$ and an injective linear map $f_2: \F^z \to \F^y$, then we obtain a corresponding factorization of $f$ with $z$ equal to the rank of $h$. Hence, a factorization of $f$ exists and the minimal $z$ is the rank of $h$, which is $\leqslant x$. Moreover, any factorization of $f$ with minimal $z$ must be obtained in this way from a factorization of $h$ with $p$ surjective and $f_2$ injective. The uniqueness of $f_2$ given $p$ is clear from the surjectivity of $p$, and the uniqueness of $p$ up to linear automorphisms of $\F^z$ follows from the observation that the kernel of $p$ must be equal to the kernel of $h$.

Condition (C4) is clear from the characterization of the minimal $z$ as the rank of $h$.
\end{example}

We remind the reader that in the following corollary, the characteristic of $\bk$ is 0.

\begin{corollary}  \label{examples are koszul}
Let $\Gamma$ be a finite group, and $d$ an integer $\geqslant 1$.
\begin{enumerate}
\item The $\bk$-linearizations of the following categories and their opposites are Koszul:
\begin{equation*}
 \cFI_\Gamma, \quad \cOI_\Gamma, \quad \cFI_d, \quad \cOI_d, \quad \cFS_\Gamma, \quad \cOS_\Gamma, \quad \cVI.
\end{equation*}

\item If $\Gamma$ is abelian, the $\bk$-linearizations of the category $\cFI^\prime_\Gamma$ (of Example \ref{egFIoG}) and its opposite are Koszul.
\end{enumerate}
\end{corollary}
\begin{proof}
(1) This is immediate from the above examples, Proposition \ref{opposite category}, and Proposition \ref{special decomposition}.

(2) For a finite abelian group $\Gamma$, the categories $\cFI^\prime_\Gamma$ and $\cFI_\Gamma$ have the same essential subcategories, so the result follows from (1) and Theorem \ref{main-theorem-2} (see Remark \ref{essential subcategory}). The opposite is Koszul by Proposition \ref{opposite category}.
\end{proof}

\section{Twists of categories over $\cFI$} \label{over FI}

In this section, we study the quadratic dual of $\cFI_\Gamma$, $\cOI_\Gamma$, $\cFI_d$ and $\cOI_d$. We shall also show that the bounded derived category of  finite dimensional graded modules of  $\cFI_\Gamma$, $\cOI_\Gamma$, $\cFI_d$ or $\cOI_d$ over a field of characteristic 0 is self-dual.

\subsection{Determinant of a finite set}
Suppose $\Delta$ is a finite set with $n$ elements, say $\Delta=\{d_1,\ldots, d_n\}$. We denote by $\bk\Delta$ the $\bk$-vector space with basis $\Delta$, and $\det(\Delta)$ the one dimensional $\bk$-vector space $\wedge^n(k\Delta)$. In particular, if $\Delta=\emptyset$, then $\det(\Delta)=k$. If $\sigma\in S_n$ is a permutation of $[n]$, then one has $d_{\sigma(1)}\cdots d_{\sigma(n)} = \mathrm{sgn}(\sigma) d_1\cdots d_n$ in $\det(\Delta)$. If $\Delta$ and $\Theta$ are finite sets, there is a canonical isomorphism
$$ \det(\Delta)\otimes_{\bk} \det(\Theta) \to \det(\Delta\sqcup \Theta), \quad d\otimes e\mapsto de. $$
If $f:\Delta\to \Theta$ is an injection, then we have an isomorphism
$$ \det(\Delta) \to \det(f(\Delta)) : d \mapsto  f(d), $$
where $f(d)=f(d_1)\cdots f(d_n)$ if $d=d_1\cdots d_n$.

\subsection{Twist construction}
Let $\C$ be any category and let $\rho: \C \to \cFI$ be any functor. In particular, if $x, y\in \Ob(\C)$ and $\rho(x) > \rho(y)$, then $\C(x,y)=\emptyset$. The category $\tC$ has a grading where $\tC(x,y)$ is in degree $\rho(y)-\rho(x)$.

Recall that for any map $f:[i]\to [j]$ (where $i,j\in\Z_+$), we write $\Delta_f$ for the set $[j]\setminus \mathrm{Im}(f)$. If $\alpha\in \C(x,y)$ and $\beta\in \C(y,z)$, then one has
$$ \Delta_{\rho(\beta\alpha)} = \rho(\beta)(\Delta_{\rho(\alpha)})\sqcup \Delta_{\rho(\beta)}. $$
We shall define a $\bk$-linear category $\tC^{tw}$, called the \emph{twist} of $(\C,\rho)$, as follows. Let $\Ob(\tC^{tw}) = \Ob(\C)$. For any $x, y\in \Ob(\C)$, let
$$ \tC^{tw}(x,y) = \bigoplus_{\alpha\in \C(x,y)} k\alpha\otimes_{\bk} \det(\Delta_{\rho(\alpha)}), $$
where $\bk\alpha$ is the one dimensional $\bk$-vector space with basis $\alpha$. The composition of morphisms
$$  \tC^{tw}(y,z) \otimes_{bk} \tC^{tw}(x,y) \to \tC^{tw}(x,z) $$
is defined by extending bilinearly the assignment
$$ (\beta\otimes e) \otimes (\alpha\otimes d) \mapsto \beta\alpha \otimes \rho(\beta)(d)e, $$
where $\alpha\in \C(x,y)$, $\beta\in C(y,z)$, $d\in \det(\Delta_{\rho(\alpha)})$, and $e\in \det(\Delta_{\rho(\beta)})$. It is easy to see that the composition of morphisms in $\tC^{tw}$ is associative. Moreover, the $\bk$-linear category $\tC^{tw}$ has a grading where $\tC^{tw}(x,y)$ is in degree $\rho(y)-\rho(x)$.

Observe that for any $x\in \Ob(\C)$, we have an isomorphism
$$\omega_{\C, x}: \tC(x,x) \stackrel{\simeq}{\to} \tC^{tw}(x,x)$$
where $\omega_{\C, x}: \alpha\mapsto \alpha\otimes 1$ for all $\alpha \in \C(x,x)$. Suppose $x,y\in \Ob(\C)$. Let us regard $\tC^{tw}(x,y)$ as a $(\tC(y,y),\tC(x,x))$-bimodule via $\omega_{\C,x}$ and $\omega_{\C,y}$. If $\rho(y)-\rho(x)=1$, we have an isomorphism of $(\tC(y,y),\tC(x,x))$-bimodules
$$\omega_{\C, x,y}: \tC(x,y) \stackrel{\simeq}{\to} \tC^{tw}(x,y)$$
where $\omega_{\C,x,y}: \alpha\mapsto \alpha\otimes d_{\alpha}$ for all $\alpha\in \C(x,y)$, and $d_\alpha$ denotes the unique element of $\Delta_{\rho(\alpha)}$.

Suppose we have a category $\D$ and a functor $\pi: \D \to \C$. Let $\psi: \D \to \cFI$ be the composition $\rho\circ\pi$, and $\tD^{tw}$ the twist of $(\D,\psi)$. We define a morphism $\pi^{tw}: \tD^{tw} \to \tC^{tw}$ by extending linearly the assignment
$$ \alpha\otimes d \mapsto \pi(\alpha)\otimes d $$
for any $\alpha\in \Mor(\D)$ and $d\in \det(\Delta_{\psi(\alpha)})$.

\subsection{The functors $\tau$ and $\mu$}
Let $\C$ be any category and let $\rho: \C \to \cFI$ be any functor. We shall define a functor
$$ \tau : \tC\Module \to \tC^{tw}\Module . $$
If $M$ is a $\tC$-module, let
$$ \tau(M)(x) = M(x) \otimes_{\bk} \det([\rho(x)]), \qquad \mbox{ for } x\in \Ob(\C). $$
The $\tC^{tw}$-module structure on $\tau(M)$ is defined by extending bilinearly to
$$ \tC^{tw}(x,y) \otimes_{\bk} \tau(M)(x) \to \tau(M)(y) $$
the assignment
$$ (\alpha\otimes d) \otimes  (v\otimes t) \mapsto \alpha(v) \otimes \rho(\alpha)(t)d, $$
where $\alpha\in \C(x,y)$, $v\in M(x)$, $d\in \det(\Delta_{\rho(\alpha)})$, and $t\in \det([\rho(x)])$. If $f: M\to N$ is a morphism of $\tC$-modules, we define a morphism
$$ \tau(f) : \tau(M)  \to \tau(N) $$
of $\tC^{tw}$-modules by
$$\tau(f)(v\otimes t) = f(v) \otimes t $$
whenever $v\in M(x)$ and $t \in \det([\rho(x)])$, for any $x\in \Ob(\C)$.

If $M$ is a graded $\tC$-module, then $\tau(M)$ is also graded with
$$ \tau(M)(x)_i = M(x)_i \otimes_{\bk} \det([\rho(x)]) \quad \mbox{ for all } x\in \Ob(\C),\, i\in \Z. $$

Similarly, we define a functor
$$ \mu : \tC^{tw}\Module \to \tC\Module . $$
If $N$ is a $\tC^{tw}$-module, let
$$ \mu(N)(x) = N(x) \otimes_{\bk} \det([\rho(x)]) ,\qquad  \mbox{ for } x\in \Ob(\C). $$
The $\tC$-module structure on $\mu(N)$ is defined by extending bilinearly to
$$ \tC(x,y) \otimes_{\bk} \mu(N)(x) \to \mu(N)(y) $$
the assignment
$$ \alpha \otimes  (v\otimes t) \mapsto (\alpha\otimes d)(v) \otimes \rho(\alpha)(t)d, $$
where $\alpha\in C(x,y)$, $v\in N(x)$, $t\in \det([\rho(x)])$, and $d=d_1\cdots d_n \in \det(\Delta_{\rho(\alpha)})$ if $\Delta_{\rho(\alpha)}=\{d_1,\ldots, d_n\}$; this is independent of the choice of ordering of $d_1,\ldots, d_n$. If $f: M\to N$ is a morphism of $\tC^{tw}$-modules, we define a morphism
$$ \mu(f) : \mu(M) \to \mu(N) $$
of $\tC$-modules by
$$ \mu(f)(v\otimes t) = f(v) \otimes t $$
whenever $v\in M(x)$ and $t \in \det([\rho(x)])$, for any $x\in \Ob(\C)$.

If $N$ is a graded $\tC^{tw}$-module, then $\mu(N)$ is also graded with
$$ \mu(N)(x)_i = N(x)_i \otimes_{\bk} \det([\rho(x)])  \quad \mbox{ for all } x\in \Ob(\C),\, i\in \Z. $$

\begin{proposition} \label{tau and mu}
The functor $\tau: \tC\Module\to \tC^{tw}\Module$ is an equivalence of categories with quasi-inverse functor $\mu$.
\end{proposition}
\begin{proof}
Let $M\in\tC\Module$ and $N\in\tC^{tw}\Module$. We have the isomorphisms
\begin{gather*}
M \stackrel{\sim}{\to} \mu(\tau(M)), \quad v \mapsto (v\otimes t) \otimes t,\\
N \stackrel{\sim}{\to} \tau(\mu(N)), \quad w \mapsto (w\otimes t) \otimes t,
\end{gather*}
where $v\in M(x)$, $w\in N(x)$, and $t=t_1\cdots t_{\rho(x)} \in \det([\rho(x)])$, for any $x\in \Ob(\C)$, and any ordering $t_1, \ldots, t_{\rho(x)}$ of the elements of $[\rho(x)]$.
\end{proof}
Observe that $\tau$ and $\mu$ also give equivalences between the corresponding categories of graded modules, locally finite modules, lower bounded modules, and graded finite dimensional modules.

\subsection{Free EI categories}
Before we prove Theorem \ref{main-theorem-4}, let us recall some basic facts on free EI categories following \cite{Li1}.

\begin{definition}
(\cite[Definition 2.1]{Li1} An \emph{EI quiver} $Q$ is a datum $(Q_0, Q_1, s, t, f, g)$ where: $(Q_0, Q_1, s, t)$ is an acyclic quiver with vertex set $Q_0$, arrow set $Q_1$, source map $s$, and target map $t$. The map $f$ assigns a group $f(x)$ to each vertex $x \in Q_0$; the map $g$ assigns an $(f(t(\alpha)), f(s(\alpha)))$-biset to each arrow $\alpha \in Q_1$.
\end{definition}

For each EI quiver $Q = (Q_0, Q_1, s, t, f, g)$, we construct an EI category $\cF_Q$ as follows. Let $\Ob (\cF_Q) = Q_0$. For any object $x \in Q_0$, let $\cF_Q (x,x) = f(x)$. For any path
\begin{equation*}
\gamma : \;  \xymatrix{x_0 \ar[r]^-{\alpha_1} & x_1 \ar[r]^-{\alpha_2} & \ldots \ar[r]^-{\alpha_n} & x_n} ,
\end{equation*}
where $n\geqslant 1$ and $\alpha_1, \ldots, \alpha_n\in Q_1$, let
\begin{equation*}
H_{\gamma}=g(\alpha_n) \times_{f(x_{n-1})} g(\alpha_{n-1}) \times_{f(x_{n-2})} \cdots \times_{f(x_1)} g(\alpha_1).
\end{equation*}
For any objects $x, y\in Q_0$ with $x\neq y$, let $\cF_Q (x,y) = \bigsqcup _{\gamma} H_{\gamma}$ where the disjoint union is taken over all paths $\gamma$ from $x$ to $y$. The composition of morphisms in $\cF_Q$ is defined in the obvious way.

\begin{definition}
An EI category $\C$ is a \emph{free EI category} if it is isomorphic to $\cF_Q$ for some EI quiver $Q$.
\end{definition}

Let $Q=(Q_0,Q_1,s,t,f,g)$ be an EI quiver.

\begin{definition}
We define the \emph{opposite} EI quiver $Q^{\op}$ of $Q$ by
\begin{equation*}
 Q^{\op}  =(Q^{\op}_0, Q^{\op}_1, s^{\op}, t^{\op}, f^{\op}, g^{\op})
\end{equation*}
where
\begin{equation*}
 Q^{\op}_0 = Q_0, \quad Q^{\op}_1 = Q_1, \quad s^{\op} = t, \quad t^{\op} = s, \quad g^{\op} = g,
\end{equation*}
and for any $x\in Q_0$,
\begin{equation*}
f^{\op}(x) = f(x)^{\op} .
\end{equation*}

\end{definition}
It is plain that $(\cF_Q)^{\op} = \cF_{Q^{\op}}$.

\subsection{Twist and quadratic duality}
\emph{Suppose that $\bk$ is a field of characteristic 0 and $\C$ is an EI category satisfying conditions (E1)--(E4) of Subsection \ref{subsection on combinatorial conditions}.}

For any $x\in\Z_+$, let us denote by $G_x$ the group $\C(x,x)$. Let $Q$ be the EI quiver
\begin{equation*}
0 \rightarrow 1 \rightarrow 2 \rightarrow 3 \rightarrow \cdots
\end{equation*}
which assigns to vertex $x$ the group $G_x$ and to arrow $x\to x+1$ the $(G_{x+1},G_x)$-biset $\C(x,x+1)$. Let $\widetilde{\C}$ be the free EI category $\cF_Q$, and denote by $\pi: \widetilde{\tC} \to \tC$ the canonical functor.

We omit the proof of the following lemma from \cite{Li2}.

\begin{lemma} \cite[Lemma 6.1]{Li2}  \label{kernel of free ei cover}
Let $x,y\in \Z_+$. The kernel of $\pi: \widetilde\tC(x,y)\to \tC(x,y)$ is spanned by the set of all elements of the form $\alpha-\beta$ where $\alpha, \beta \in \widetilde\C(x,y)$ are morphisms such that $\pi(\alpha)=\pi(\beta)$.
\end{lemma}

Recall from Subsection \ref{subsection on quadratic categories} that $\hat\tC$ denotes the free cover of $\tC$.

\begin{lemma}  \label{free cover vs free ei cover}
There are natural isomorphisms $\hat\tC \cong \widetilde\tC$ and $\check{\tC} \cong (\widetilde\tC)^{\op}$.
\end{lemma}
\begin{proof}
The first isomorphism is immediate from constructions. The second isomorphism follows from the observations that $\tC(x,x)$ is naturally isomorphic to $\tC(x,x)^{\op}$, and one has natural isomorphisms
\begin{equation*}
   \Hom_{\tC(y,y)} (\tC(x,y), \tC(y,y) )\; \cong\; \Hom_{\bk} (\tC(x,y), \bk) \; \cong\;  \tC(x,y) .
 \end{equation*}
\end{proof}

Now suppose we have a functor $\rho: \C \to \cFI$ such that $\rho(x)=x$ for all $x\in\Z_+$.

It follows from Lemma \ref{kernel of free ei cover} that the kernel of $\pi^{tw}: \widetilde\tC^{tw} \to \tC^{tw}$ is spanned by $(\alpha-\beta)\otimes d$ for all $\alpha, \beta\in \hat\C$ such that $\pi(\alpha)=\pi(\beta)$, and where $d=d_1\cdots d_n \in \det(\Delta_{\rho(\pi(\alpha))})$ for any ordering $d_1,\ldots, d_n$ of the elements of $\Delta_{\rho(\pi(\alpha))}$. Moreover, if the kernel of $\pi: \widetilde\tC \to \tC$ is generated by its degree 2 elements, then the kernel of $\pi^{tw}: \widetilde\tC^{tw} \to \tC^{tw}$ is generated by its degree 2 elements.

There is a unique morphism
\begin{equation*}
\omega: \widetilde\tC \longrightarrow \widetilde\tC^{tw}
\end{equation*}
such that $\omega$ is the identity map on the set of objects, its restriction to $\widetilde\tC(x,x)$ is $\omega_{\widetilde\C,x}$ for all $x\in \Ob(\C)$, and its restriction to $\widetilde\tC(x,y)$ is $\omega_{\widetilde\tC, x,y}$ for all $x,y\in \Ob(\C)$ with $\rho(y)-\rho(x)=1$. It is clear that $\omega$ is bijective, and hence an isomorphism. Moreover, the composition $\pi^{tw}\circ\omega: \widetilde\tC \to \tC^{tw}$ is the free cover of $\tC^{tw}$. Thus, if $\tC$ is quadratic, then $\tC^{tw}$ is quadratic.

Now if $\C=\cFI_\Gamma$ or $\cOI_\Gamma$, we let $\rho:\C\to \cFI$ be the functor sending $(f,c)\in \C$ to $f\in\cFI$. If $\C=\cFI_d$ or $\cOI_d$, we let $\rho:\C\to \cFI$ be the functor sending $(f,\delta)$ to $f\in \cFI$.

\begin{theorem} \label{dual of fi}
Let $\C$ be $\cFI_\Gamma$, $\cOI_\Gamma$, $\cFI_d$, or $\cOI_d$. Then $\tC^!$ is isomorphic to $(\tC^{tw})^{\op}$.
\end{theorem}

\begin{proof}
Let $\C$ be $\cFI_\Gamma$. Let $x\in \Z_+\setminus\{0\}$. We have the composition map
$$ \eta: \tC(x,x+1) \otimes_{\tC(x,x)} \tC(x-1,x) \to \tC(x-1,x+1) $$
and its pullback
$$ \eta^* : \tC(x-1,x+1)^* \to \tC(x-1,x)^* \otimes_{\tC(x,x)} \tC(x,x+1)^*. $$
Equivalently, we have
$$ \eta^* : \tC(x-1,x+1) \to \tC(x-1,x) \otimes_{\tC(x,x)^{\op}} \tC(x,x+1); $$
see Lemma \ref{free cover vs free ei cover}.

For each $r\in [x+1]$, we choose any $f_r\in \C(x,x+1)$ such that the image of $\rho(f_r):[x]\to [x+1]$ is $[x+1]\setminus\{r\}$. We have an isomorphism of right $\tC(x,x)$-modules
$$ \tC(x,x)^{\oplus (x+1)} \stackrel{\simeq}{\to} \tC(x,x+1),\quad (g_1,\ldots,g_{x+1}) \mapsto f_1g_1+\cdots+f_{x+1}g_{x+1}. $$
Therefore, we have a linear bijection
\begin{gather*}
\tC(x-1,x)^{\oplus (x+1)} \stackrel{\simeq}{\to} \tC(x,x+1) \otimes_{\tC(x,x)} \tC(x-1,x),\\
(h_1,\ldots,h_{x+1}) \mapsto f_1\otimes h_1 + \cdots + f_{x+1} \otimes h_{x+1}.
\end{gather*}

Suppose $f\in \C(x-1,x+1)$. The set $\Delta_{\rho(f)}$ has exactly two elements. If $r\notin \Delta_{\rho(f)}$, then $f\neq f_r\circ h$ for all $h\in \C(x-1,x)$. If $r\in \Delta_{\rho(f)}$, then there exists a unique $h_r\in \C(x-1,x)$ such that $f=f_r\circ h_r$. Let us write $\Delta_{\rho(f)} = \{r, s\}$. Then
$$ \eta^*(f) = h_r \otimes f_r + h_s \otimes f_s. $$
On the other hand, one has
\begin{multline*}
\omega(f_r\otimes h_r + f_s \otimes h_s) = (f_r\otimes h_r)\otimes (s\wedge r) + (f_s\otimes h_s)\otimes (r\wedge s) \\
= ( f_r\otimes h_r  - f_s\otimes h_s )\otimes (s\wedge r).
\end{multline*}
Hence, the image of $\eta^*$ is precisely the kernel of
$$ (\pi^{tw} \circ \omega)^{\op} : \hat\tC^{\op}(x-1,x+1) \to (\tC^{tw})^{\op}(x-1,x+1). $$
By Proposition \ref{Koszul categories are quadratic} and Corollary \ref{examples are koszul}, $\tC$ is quadratic, and hence $(\tC^{tw})^{\op}$ is quadratic. It follows that $\tC^! = (\tC^{tw})^{\op}$.

The proofs when $\C$ is $\cOI_\Gamma$, $\cFI_d$ or $\cOI_d$ are similar.
\end{proof}

We do not know of a similar description of the quadratic dual of the other examples in Corollary \ref{examples are koszul}.

\begin{corollary} \label{morita to yoneda}
Let $\C$ be $\cFI_\Gamma$, $\cOI_\Gamma$, $\cFI_d$, or $\cOI_d$. Let $\tY$ be the Yoneda category of $\tC$. Then the categories $\tC\Module$ and $\tY\Module$ are equivalent.
\end{corollary}
\begin{proof}
This is immediate from Proposition \ref{yoneda vs quadratic dual}, Corollary \ref{examples are koszul}, Proposition \ref{tau and mu}, and Theorem \ref{dual of fi}.
\end{proof}

\subsection{Self-duality functor}
In \cite[Section 6]{SS0}, S. Sam and A. Snowden constructed a self-duality functor (called the Fourier transform) on the bounded derived category of finitely generated $\FI$-modules over a field of  characteristic 0; see also \cite[(3.3.8)]{SS1}. We follow their idea in the following construction.

\emph{Suppose that  $\bk$ is a field of characteristic 0, and $\C$ is $\cFI_\Gamma$, $\cOI_\Gamma$, $\cFI_d$, or $\cOI_d$. }

From Corollary \ref{koszul duality functor on bounded derived categories}, Corollary \ref{examples are koszul}, Proposition \ref{tau and mu}, and Theorem \ref{dual of fi}, we have a composition of equivalences:
\begin{align*}
 D^b(\tC\fdmod) & \stackrel{\mathrm{K}}{\longrightarrow} D^b(\tC^!\fdmod) \\
                      & \stackrel{\simeq}{\longrightarrow}  D^b( (\tC^{tw})^{\op}\fdmod ) \\
                      & \stackrel{(-)^*}{\longrightarrow} D^b( \tC^{tw}\fdmod )^{\op} \\
                      & \stackrel{\mu}{\longrightarrow} D^b( \tC\fdmod )^{\op},
\end{align*}
where $\mathrm{{gmod}_{fd}}$ denotes the category of graded modules which are finite dimensional. Let us denote this composition by $\kappa$. We also have:
\begin{align*}
D^b(\tC\fdmod) & \stackrel{\tau}{\longrightarrow}  D^b( \tC^{tw}\fdmod ) \\
& \stackrel{(-)^*}{\longrightarrow}   D^b( (\tC^{tw})^{\op}\fdmod )^{\op} \\
& \stackrel{\simeq}{\longrightarrow}  D^b(\tC^!\fdmod)^{\op} \\
& \stackrel{\mathrm{K'}}{\longrightarrow}  D^b( \tC\fdmod )^{\op},
\end{align*}
where $\mathrm{K'}$ is the quasi-inverse of $\mathrm{K}$ defined in \cite[Theorem 2.12.1]{BGS} and \cite[(5.6)]{Mazorchuk}. Let us denote this composition of functors by $\kappa'$. We have $\kappa\,\kappa' \cong \Id$ and $\kappa' \kappa \cong \Id$. This, together with Theorem \ref{dual of fi} and Corollary \ref{morita to yoneda}, proves Theorem \ref{main-theorem-4}.

\begin{remark}
For a complex $M$ of graded vector spaces, the vector space duality functor $(-)^*$ sends $M$ to the complex $L$ with $L^i_j = (M^{-i}_{-j})^*$ and $d^i_L(f) = (-1)^i f d_M$ for $f\in L^i_j$.
\end{remark}

\begin{proposition}
One has:
\begin{enumerate}
\item  $\kappa \cong \kappa'$.

\item $\kappa^2 \cong \Id$.
\end{enumerate}
\end{proposition}
\begin{proof}
(1) Let $\varpi = (-)^*\kappa$ and $\varpi' = (-)^*\kappa'$. Let $M\in D^b(\tC\fdmod)$.  From the proof of \cite[Theorem 2.12.1]{BGS}, we have the isomorphisms
\begin{align*}
\varpi(M)^p(y)_q  &\stackrel{\cong}{\longrightarrow}  \bigoplus_{ \substack{ i+j=p \\ x-j = y+q } }
\left(  \tC^{tw}(y,x) \otimes_{\tC(x,x)} M^i(x)_j \right)  \otimes_{\bk} \det([\rho(y)]) \\
 &\stackrel{\cong}{\longrightarrow}  \bigoplus_{ \substack{ i+j=p \\ x-j = y+q } }
 \tC(y,x) \otimes_{\tC(x,x)} \left( M^i(x)_j \otimes_{\bk} \det( [\rho(x)] ) \right) \\
&\stackrel{\cong}{\longrightarrow} \varpi' (M)^p(y)_q .
\end{align*}
It is straightforward to verify that the composition of the above isomorphisms is compatible with the $\tC$-module structures. It is also compatible with the differentials up to signs, but by a routine verification, the complex is isomorphic to the one with the altered signs.\footnote{The isomorphism is $(-1)^{ij}$.} Thus, $\varpi \cong \varpi'$, and hence $\kappa\cong\kappa'$.

(2) One has: $\kappa^2 \cong \kappa' \kappa \cong \Id$.
\end{proof}


\begin{thebibliography}{99}


\bibitem{BGS} A. Beilinson, V. Ginzburg, W. Soergel, \textit{Koszul duality patterns in representation theory}, J. Amer. Math. Soc., 9 (1996), 473-527.

\bibitem{Broue} M. Brou\'{e}, \textit{Higman's criterion revisited}, Michigan Math. J. 58 (2009), 125-179.

\bibitem{CF} T. Church, B. Farb, \textit{Representation theory and homological stability}, Adv. Math. 245 (2013), 250-314, arXiv:1008.1368.

\bibitem{CEF} T. Church, J. Ellenberg, B. Farb, \textit{FI-modules and stability for representations of symmetric groups}, Duke Math. J. 164 (2015), no. 9, 1833-1910, arXiv:1204.4533.

\bibitem{CEFN}  T. Church, J. Ellenberg, B. Farb, R. Nagpal, \textit{$\mathrm{FI}$-modules over Noetherian rings}, Geom. Top. 18-5 (2014), 2951-2984, arXiv:1210.1854.

\bibitem{Dieck} T. tom Dieck, \textit{Transformation groups}, de Gruyter Studies in Math. 8, Walter de Gruyter, (1987).

\bibitem{Djament} A. Djament, \textit{Des propri\'{e}t\'{e}s de finitude des foncteurs polynomiaux}, Fundamenta Mathematicae 233 (2016), 197-256, arXiv:1308.4698.


\bibitem{Drozd} Y. Drozd, V. Mazorchuk, \textit{Koszul duality for extension algebras of standard modules}, J. Pure Appl. Algebra 211 (2007), 484-496, arXiv:math/0411528.


\bibitem{Farb}  B. Farb, \textit{Representation stability}, Proceedings of the International Congress of Mathematicians, Seoul 2014, Vol. II, 1173-1196, arXiv:1404.4065.

\bibitem{GL} W. L. Gan, L. Li, \textit{Noetherian property of infinite EI categories},  New York J. Math. 21 (2015), 369-382, arXiv:1407.8235.


\bibitem{KM} G. Kudryavtseva, V. Mazorchuk, \textit{Partialization of categories and inverse braid-permutation monoid}, Internat. J. Algebra Comput. 18 (2008), no. 6, 989-1017, arXiv:math/0610730.

\bibitem{Li1} L. Li, \textit{A characterization of finite EI categories with hereditary category algebras}, J. Algebra 345 (2011), 213-241, arXiv:1103.0959.

\bibitem{Li2} L. Li, \textit{A generalized Koszul theory and its application}, Trans. Amer. Math. Soc. 366 (2014), 931-977, arXiv:1109.5760.

\bibitem{Li3} L. Li, \textit{Derived equivalences between triangular matrix algebras}, Comm. Algebra 46 (2018), 615-628.

\bibitem{Luck} W. L\"{u}ck, \textit{Transformation groups and algebraic K-theory}, Lecture Notes in Mathematics 1408, Springer-Verlag, (1989).



\bibitem{Mazorchuk} V. Mazorchuk, S. Ovsienko, C. Stroppel, \textit{Quadratic duals, Koszul dual functors, and applications}, Trans. Amer. Math. Soc. 361 (2009), 1129-1172, arXiv:math/0603475.

\bibitem{PS} A. Putman, S. Sam, \textit{Representation stability and finite linear groups}, arXiv:1408.3694.


\bibitem{RS} B. Richter, B. Shipley, \textit{An algebraic model for commutative HZ-algebras}, arXiv:1411.7238.

\bibitem{SS0} S. Sam, A. Snowden, \textit{GL-equivariant modules over polynomial rings in infinitely many variables}, Trans. Amer. Math. Soc. 368 (2016), 1097-1158,  arXiv:1206.2233.

\bibitem{SS1} S. Sam, A. Snowden,  \textit{Stability patterns in representation theory},  Forum Math. Sigma 3 (2015), e11, 108 pp., arXiv:1302.5859.


\bibitem{SS} S. Sam, A. Snowden, \textit{Gr\"{o}bner methods for representations of combinatorial categories}, to appear in J. Amer. Math. Soc., arXiv:1409.1670.

\bibitem{SS2} S. Sam, A. Snowden, \textit{Representations of categories of G-maps}, arXiv:1410.6054.

\bibitem{Sh} U. Shukla, \textit{On the projective cover of a module and related results}, Pacific J. Math. 12 (1962), 709-717.



\bibitem{Wilson} J. Wilson, \textit{$\FI_{\mathcal{W}}$-modules and stability criteria for representations of classical Weyl groups}, J. Algebra 420 (2014), 269-332, arXiv:1309.3817.


\end{thebibliography}
\end{document}